\documentclass[oneside,notitlepage,12pt]{article}

\usepackage{amssymb}
\usepackage[leqno]{amsmath}
\usepackage{amsfonts}
\usepackage{amsopn}
\usepackage{amstext}
\usepackage{amsthm}

\usepackage{wasysym}
\usepackage{fourier-orns}
\usepackage{mathtools}
\usepackage{tikz}
\usepackage{tikz-cd}
\usetikzlibrary{cd}
\newcommand{\nx}{\pgfmatrixnextcell}
\usetikzlibrary{arrows.meta} 
\usetikzlibrary{graphs} 

\usepackage{enumitem}

\usepackage{verbatim}
\usepackage[colorlinks]{hyperref}
\usepackage{makeidx}

\newcommand{\define}[2]{{\em #1}\index{#2}}

\usepackage{calrsfs}
\usepackage{clock} 

\newcommand{\Em}{\mathbb M}

\providecommand{\cal}{\mathcal}
\renewcommand{\Bbb}{\mathbb}

\newenvironment{pf}{\begin{proof}}{\end{proof}}

\newcommand{\Ee}{{\cal{E}}}
\newcommand{\Ef}{{\cal{F}}}

\newcommand{\Tau}{{\cal{T}}}

\newcommand{\Zee}{{\Bbb{Z}}}

\newcommand{\Nat}{{\Bbb{N}}}
\newcommand{\Qyu}{{\Bbb{Q}}}
\newcommand{\Err}{{\Bbb{R}}}

\newcommand{\al}{\alpha}

\newcommand{\sig}{\sigma}
\newcommand{\eps}{\varepsilon}
\renewcommand{\phi}{\varphi}
\renewcommand{\rho}{\varrho}

\newcommand{\unii}{\mathbb I}
\newcommand{\ntr}{{n\in\omega}}

\newcommand{\loe}{\leq}
\newcommand{\goe}{\geq}

\newcommand{\subs}{\subseteq}

\newcommand{\nnempty}{\ne\emptyset}

\renewcommand{\iff}{\Longleftrightarrow}

\newcommand{\id}[1]{{\operatorname{i\!d}_{#1}}} 

\newcommand{\dom}{\operatorname{dom}}
\newcommand{\cod}{\operatorname{cod}}

\newcommand{\oraz}{\qquad\text{and}\qquad}

\newcommand{\poset}{{\Bbb{P}}}

\newcommand{\length}{\operatorname{length}}

\newtheorem{tw}{Theorem}[section]
\newtheorem{wn}[tw]{Corollary}
\newtheorem{lm}[tw]{Lemma}
\newtheorem{prop}[tw]{Proposition}
\newtheorem{claim}[tw]{Claim}
\theoremstyle{definition}
\newtheorem{df}[tw]{Definition}
\newtheorem{ex}[tw]{Example}

\newtheorem{problem}[tw]{Problem}

\theoremstyle{remark}

\newcommand{\set}[1]{\{#1\}}
\newcommand{\setof}[2]{\{#1\colon #2\}}

\newcommand{\seq}[1]{\langle #1 \rangle}

\newcommand{\sett}[2]{\{#1\}_{#2}}
\newcommand{\sn}[1]{\{#1\}} 
\newcommand{\dn}[2]{\{#1,#2\}} 
\newcommand{\pair}[2]{\langle #1, #2 \rangle} 
\newcommand{\triple}[3]{\langle #1, #2, #3 \rangle} 
\newcommand{\map}[3]{#1\colon #2 \to #3} 
\newcommand{\umap}[3]{#2 \stackrel{#1}{\to} #3} 
\newcommand{\img}[2]{#1[#2]} 

\newcommand{\fra}{Fra\"iss\'e}

\newcommand{\rless}{\prec}
\newcommand{\rmore}{\succ}

\providecommand{\nat}{\omega}

\newcommand{\ciag}[1]{{\sett{{#1}_n}{\ntr}}}

\newcommand{\aut}{\operatorname{Aut}}

\newcommand{\ob}[1]{\operatorname{Obj}( #1 )}
\newcommand{\iso}{\approx}

\newcommand{\ciagi}[1]{\sig{#1}}

\newcommand{\fK}{{\mathfrak{K}}}

\newcommand{\fC}{{\mathfrak{C}}}

\newcommand{\bX}{{\mathbb{X}}}

\newcommand{\cmp}{\circ} 

\newcommand{\Pel}{\mathbf P}

\newcommand{\BM}[1]{\operatorname{BM}\left(#1\right)}
\newcommand{\BMG}[2]{\BM {#1, #2} }

\newcommand{\separator}{\begin{center} \leafright \leafright \leafright \decotwo \decotwo \decotwo \leafleft \leafleft\leafleft
\end{center}}

\newcommand{\iza}{\Theta} 

\newcommand{\mrak}[1]{{\color{blue} #1 }}
\newcommand{\koment}[1]{}

\providecommand{\ar}{\arrow}
\newcommand{\bP}{\mathbb{P}}

\newcommand{\catgen}[1]{\langle #1 \rangle}

\newcommand{\cTau}{\catgen{\Tau}}
\newcommand{\size}{\operatorname{size}}
\newcommand{\uv}{\mathfrak{V}}
\newcommand{\eva}[1]{{#1}^{\operatorname{fin}}}
\newcommand{\Eva}[1]{{#1}^{\sigma}}

\newcommand{\Iso}{\operatorname{Iso}}

\newcommand{\pustka}[1]{} 

\newcommand{\Kat}{Kat\v{e}tov}

\usepackage{mathdots}
\usetikzlibrary{decorations.pathmorphing}
\usetikzlibrary{decorations.markings}

\tikzset{
	ppath/.style = {
		dash pattern = on (#1 * 7pt) off 2pt,
		postaction = {
			decorate,
			decoration = {
				markings,
				mark = between positions (#1 * 7pt + 0.2pt) and 1 step (#1 * 7pt + 2pt) with {\arrow{>};},
			},
		},
	},
	path/.default = 1,
}

\title{Evolution systems: A framework for studying generic mathematical structures}
\author{
{\sc Wies{\l}aw Kubi\'s}\footnote{Institute of Mathematics, Czech Academy of Sciences, Czech Republic. Research supported by EXPRO project 20-31529X (Czech Science Foundation).}
\and
{\sc Paulina Radecka}\footnote{Institute of Mathematics, Czech Academy of Sciences, Czech Republic and Warsaw University of Technology, Poland. Research supported by EXPRO project 20-31529X (Czech Science Foundation).}
}

\date{\clocktime\today}

\makeindex

\begin{document}

\maketitle

\begin{abstract}
	We introduce the concept of an evolution system, which provides a convenient framework for studying generic mathematical structures and their properties.
	Roughly speaking, an \emph{evolution system} is a category endowed with a selected class of morphisms called \emph{transitions}, and with a selected object called the \emph{origin}.
	We illustrate it by a series of examples from several areas of mathematics. We formulate sufficient conditions for the existence of the unique ``most complicated'' evolution.
	In case the evolution system ``lives'' in model theory and nontrivial transitions are one-point extensions, the limit of the most complicated evolution is known under the name \emph{\fra\ limit}, a unique countable universal homogeneous model determined by a fixed class of finitely generated models satisfying some obvious axioms.
	
	Evolution systems can also be viewed as a generalization of abstract rewriting systems, where the partially ordered set is replaced by a category. In our setting, the \emph{process} of rewriting plays an important role, whereas in rewriting systems only the result of a rewriting procedure is relevant. An analogue of Newman's Lemma holds in our setting, although the proof is a bit more delicate, nevertheless, still based on Huet's idea using well founded induction.
	
\ \\
\noindent
{MSC (2010):}
18A05, 
03C95, 
18A30. 

\noindent
{\it Keywords:} Evolution system, evolution, amalgamation, absorption property, confluence, determination, termination.
\end{abstract}

\tableofcontents

\section{Introduction}

We start with a loose explanation of our main ideas, followed by a series of motivating examples. Let us imagine that a certain transition system is given, starting from some initial state $\iza$ and having the \emph{amalgamation property}, saying that for every two transitions $f$, $g$ from the same state it is possible to make several further transitions $f_1,\dots, f_n$, $g_1,\dots,g_m$ so that the compositions
$$f_n \cmp \dots \cmp f_1 \cmp f \oraz g_m \cmp \dots \cmp g_1 \cmp g$$
are the same (or almost the same, in case some measuring of their distance is allowed), in particular, leading to the common state. It is then natural to expect that there might exist a special infinite process (evolution) accumulating all possible states and in some sense recording all possible transitions. Specifically, denoting a fixed transition $f$ from $A$ to $B$ by $\umap{f}{A}{B}$, an infinite evolution can be represented as an infinite sequence of the form
$$\begin{tikzcd}
A_0 \ar[r, "f_0"] & A_1 \ar[r, "f_1"] & A_2 \ar[r, "f_2"] & \cdots
\end{tikzcd}$$
where $A_0$ is the initial state $\iza$.  Saying that such a process has the \emph{absorption property} for all transitions means that for every $n$, given a transition $\umap{f}{A_n}{Y}$ going out of the process, there exist $m>n$ and a sequence of transitions 
$$\begin{tikzcd}
	Y \ar[r, "g_0"] & Y_1 \ar[r, "g_1"] & \cdots \ar[r, "g_k"] & Y_{k+1} = A_m
\end{tikzcd}$$
such that the composition $f_{m-1} \cmp \dots \cmp f_n$ is the same as $g_k \cmp \dots \cmp g_0 \cmp f$.
It turns out that if a process with the absorption property exists, it is essentially unique. This means that every two processes with the absorption property are isomorphic, which in turn means that there is a way of ``jumping" between the first and the second one infinitely many times.

It is rather evident that the proper framework for describing and studying evolution systems comes from category theory. Namely, the objects are the system states while the arrows are transitions or, more generally, compositions of transitions. Category theory offers elegant and quite strong, yet at the same time manageable, tools. A good example is the concept of a functor: in order to describe an infinite evolution, it suffices to use only functors from the set of non-negative integers $\nat$ (more traditionally denoted by $\Nat$) viewed as a category in which the objects are natural numbers and arrows are pairs of the form $\pair n m$ where $n \loe m$. 

Actually, a significant power of category theory lies within the notion of \emph{colimit} of a functor, unifying concepts like \emph{supremum} in partially ordered sets, unions of families of structures, and so on. In particular, every functor from the natural numbers has its colimit in a suitable, possibly bigger, category.

While it is possible to investigate infinite evolutions in their ``pure'' form, it is more convenient and more natural to look at their colimits, identifying them with the isomorphism classes of certain objects in a bigger category. As mentioned above, actually we get an isomorphism class of a single object, which typically has many symmetries. This is due to \emph{homogeneity}, saying that every transition between two states of evolution processes with the absorption property can be extended to (or at least approximated by) an isomorphism between these processes.

Our goal is to introduce and study {abstract evolution systems}, focusing on evolutions with the absorption property. In particular, we show obvious connections with the theory of universal homogeneous structures, known in model theory as \emph{\fra\ limits}.

\separator

The note is organized as follows. We start with motivating examples (Section~\ref{SecTwoo}) and we introduce formal definitions in Section~\ref{SecThreee} together with examples illustrating the theory. Next we present the theory of universal homogeneous structures (Section~\ref{SecFourr}), with a result involving a natural infinite game.
Section~\ref{SecRwrtngsTrms} discusses determined and terminating evolution systems, proving an analogue of Newman's Lemma and showing that terminating systems lead to finite homogeneous objects. Section~\ref{graphrewriting} explains how graph rewriting can be viewed in the language of evolution systems. 
The last Section~\ref{SecFajw} contains a discussion of selected further research directions.

\koment{\mrak{
We are also planning to  search for new concrete applications in physics, philosophy, and related areas of science, not excluding mathematics of course, where many examples had already been explored and where the theoretical foundations live.
}}

\paragraph{Historical remarks.}
The concept of an evolution system is formally new (although the ideas are at least as old as the theory of categories), however the main results are adaptations of abstract \emph{\fra\ theory} of universal homogeneous structures. This theory was created by Roland \fra~\cite{Fraisse} around the year 1954 using the language of model theory. Namely, \fra\ observed that Cantor's theorem characterizing the linearly ordered set of the rational numbers can actually be stated and proved in a much more general setup, using any first-order language. Actually, \fra\ only considered relational structures, however adding algebraic operations does not affect the result.
\fra's main result states that, given a suitable class of finite structures satisfying certain natural conditions (including the amalgamation property), there exists a unique countable ultra-homogeneous structure from which one can reconstruct the original class. Ultra-homogeneity (often called just \emph{homogeneity}) means that every isomorphism between finite substructures extends to an automorphism.
Soon after \fra's work, the theory was extended by J\'onsson~\cite{Jon} to uncountable structures, with extra cardinal arithmetic assumptions.
As it happens, \fra\ theory actually has a purely category-theoretic nature, although it was formally stated this way (almost forty years after \fra's work) by Droste and G\"obel~\cite{DroGoe} and several years ago explored by the first author~\cite{Kub40, Kub41}, also in metric-enriched categories. The concept of an evolution system, the main topic of the present note, arose few years ago during the first author's investigations, mainly as a natural language extracting the essence of the theory of universal homogeneous structures, while also offering a wider scope of possible applications, far beyond model theory.

{Regarding homogeneity, we want to point out the connection to the notion of \textit{self-similarity} form works of Tom Leinster \cite{Lei04, Lei11}. In his paper he says "Local statements of self-similarity say something like
`almost any small pattern observed in one part of the object can be observed
throughout the object, at all scales'. [...] Global statements say something	like `the whole object consists of several smaller copies of itself glued together'". This undoubtedly corresponds to the notion of homogeneity: looking at small patterns in one part would mean that we are given an isomorphism between substructures, and "can be observed throughout the object" means that it can be extended to an automorphism of the whole structure.

Nevertheless, Leinster is using $M$-coalgebras for a suitable endofunctor of an equational system; then one looks for the universal family of spaces satisfying the equations. We believe that our language ultimately addresses the concept of self-similarity in a simpler context, without incorporating the notion of a module of a category. This makes our theory more accessible not only to readers familiar with category theory but also to other mathematicians and other scientists acquainted with the basics of category theory, e.g., some theoretical physicists and philosophers.}

Let us mention that \fra\ theory gained a lot of attention after the seminal work of Kechris, Pestov, and Todor\v{c}evi\'c~\cite{KPT} discovering a correspondence between dynamic properties of the automorphism group of the \fra\ limit and combinatorial properties of the \fra\ class.
Currently there are several lines of research exploring various faces of the theory of universal homogeneous structures, mainly in model theory (including its continuous variant dealing with metric structures) and in pure category theory.
The current note also exhibits connections between finite \fra\ limits and terminating rewrting systems.

Summarizing, mathematical objects resembling \fra\ limits appear in several areas of mathematics and one of the goals of this note is to present some of them, through the ``looking glass'' of evolution systems.

\section{Examples}\label{SecTwoo}

Before going into technical details, we wish to present several motivating examples that fit into our framework.

\subsection{Shattering glass}

Let us look at a very natural evolution system starting from a nice ``glass'' rectangle $S$ (or any other polygon made of glass), whose transitions are formed by breaking the glass into smaller and smaller pieces.
A state in this system is a finite family of pairwise disjoint polygons that can be glued together, recovering $S$.
By ``breaking'' a polygon $P$ we mean replacing it by polygons $Q_1, \dots, Q_k$ that play the role of the ``pieces'' of $P$, namely, all of them can be translated in such a way that the union is $P$ and their interiors are pairwise disjoint.
A \emph{transition} from a state $\{P_1, \dots, P_n\}$ to a state $\{Q_1, \dots, Q_m\}$ is breaking each $P_i$ (or just a selected one) and just collecting the pieces together. We need to keep track of how the pieces were originally arranged, which can be easily achieved by keeping in mind a continuous surjective mapping $\map{f}{\bigcup_{i=1}^m Q_i}{\bigcup_{j=1}^n P_j}$ performing the gluing, namely, it is an isometry on each $Q_i$ and each $P_j$ is the union of images of some $Q_i$s with pairwise disjoint interiors.
So after all, a transition from a state $\{P_1, \dots, P_n\}$ to a state $\{Q_1, \dots, Q_m\}$ is a suitable continuous function from $\bigcup_{i=1}^m Q_i$ onto $\bigcup_{j=1}^n P_j$ that `memorizes' the breaking in the sense that the function knows how to glue the pieces back.
It is rather clear that this transition system has the amalgamation property: Given two breakings of a polygon $P$, there is a very concrete breaking of $P$ refining both, simply by intersecting all the pieces.

The infinite process of breaking the initial glass polygon $S$ leads to a compact planar set $C$, the inverse limit of the sequence of ``breaking'' mappings.
If the process has the absorption property, $C$ is homeomorphic to the well known Cantor set. What is the conclusion?
Well, one can say that no matter how nice the initial glass polygon is, after breaking it thoroughly infinitely many times, we always receive the same ``dust", namely, the Cantor set whose geometry is rather mediocre, as it lacks nontrivial connected subsets.

As a conclusion, we see that there exist natural evolution processes, like the one described above, whose limits might be of different nature and outside of the real world. Nevertheless, investigating such processes and their limits may lead to a better understanding of the original evolution system.

\subsection{Ribbons}

Given a flexible ribbon, it is easy to imagine many possibilities of folding it so that after squeezing (and possibly gluing) the material we obtain another ribbon. Such a transition can be represented by a continuous surjection $\map f I J$, where $I$, $J$ are the two ribbons, namely compact intervals of the real numbers. Since each two such intervals have exactly the same structure, we may assume $I = J = [0,1]$. The process of folding and squeezing the ribbon without reverting it can be recorded as a continuous surjection $\map f{[0,1]}{[0,1]}$.
Thus, we are dealing with an evolution system where all the states are identical, however the transitions could be quite complicated.
This is indeed the case, as the natural limit of the process \text{red}{evolution} with the absorption property is the \emph{pseudo-arc}, a rather intriguing planar geometric object. Mathematically, this is the unique, up to homeomorphism, compact connected subset $\bP$ of the plane which from a far distance looks like the interval (formally it is called \emph{chainable} or \emph{arc-like}), while at the same time it cannot be written as the union of two connected proper closed subsets. Furthermore, every nontrivial closed connected subset of $\bP$ is homeomorphic to $\bP$.

The amalgamation property of this evolution system is known under the name \emph{Mountain Climbing Theorem}, saying that for each two reasonable (say, piecewise monotone) continuous surjections $f,g$ from the unit interval onto itself there exist continuous surjections $f'$, $g'$ on the unit interval satisfying $f \cmp f' = g \cmp g'$.
Assuming $f(0)=0=g(0)$, $f(1)=1=g(1)$, drawing the graphs of $f$ and $g$, we can imagine a mountain and the statement above says that two climbers can go from the bottom to the top on the two different mountain slopes in such a way that at each moment of time their altitudes are the same.

The evolution process described above exhibits the fact that sometimes the states play an inferior role to the transitions that carry all the relevant concrete information about the process, leading to rather surprising structures, again very different from the states of the system.
On the other hand, the pseudo-arc contains all the relevant information about possible continuous surjections between closed intervals.

\subsection{Simplices}

There is a clear definition of a finite-dimensional simplex: The convex hull of an affinely independent finite set. So, the 0-dimensional simplex is a point, the 1-dimensional simplex is an interval, the 2-dimensional simplex is a triangle, and so on.
All the finite-dimensional simplices could be thought of states of some evolution system. The question is how to describe transitions. The obvious possibility is to consider embeddings onto faces, namely, the $k$-dimensional simplex $\Delta_k$ can be isometrically embedded into any $\Delta_m$ with $m>k$ so that its vertices are within the vertices of $\Delta_m$. Let us assume that $m = k+1$. Then there are exactly $(k+1)! \cdot (k+2) = (k+2)!$ possibilities for such embeddings. This definitely makes sense, nevertheless every infinite process in this system is actually the same: it is a strictly increasing chain of finite-dimensional simplices in which the successor of each simplex is built by adding one more vertex in a new dimension.
It turns out that another natural transition from $\Delta_k$ to $\Delta_{k+1}$ can be a pair consisting of the embedding as above together with a fixed affine projection $\map{p}{\Delta_{k+1}}{\Delta_k}$. Note that $p$ is actually determined by choosing a point $x_p \in \Delta_k$.
In any case, now our evolution system becomes much more complicated, once we insist on recording the projections.
One can explain this approach by assuming that each simplex is actually a very flexible geometric figure, so that choosing a point inside of it one can pull it out, obtaining a more complicated simplex-like figure, recording where the point initially was. The reverse of the procedure of pulling out is affine, of course.

This evolution system clearly has the amalgamation property. An evolution with the absorption property leads to the \emph{Poulsen simplex}, the unique (up to affine homeomorphism) metrizable simplex (contained in the Hilbert space) whose set of vertices is dense.

Contrary to the previous examples, the Poulsen simplex contains all finite-dimensional simplices, in fact every inverse limit of finite-dimensional simplices with affine projections is affinely homeomorphic to a face of the Poulsen simplex. The fact that the Poulsen simplex carries more information about the evolution system than the other objects, like the Cantor set or the pseudo-arc is just illusion. An explanation is very simple: The transitions  
in the system producing the Poulsen simplex are capable of recording the history, namely, each simplex $\Delta_k$ appears as a concrete face of $\Delta_m$ for every $m>k$.
This is not the case in the previous example of glass polygons, where the transitions actually change the particles, by breaking them into smaller ones. 

This example is treated in detail in the work of Kwiatkowska and the first author (see \cite{KubKw}). For a categorical treatment of simplices we refer to Gabriel and Zisman's monograph~\cite{GZ67}.

\section{Preliminaries}\label{SecThreee}

We adapt the convention which many category-theorists support, namely, that arrows are more important than objects. Thus, a category $\fC$ will be identified with its class of arrows and $\fC(A,B)$ will denote the set of all $\fC$-arrows with domain $A$ and codomain $B$. For our purposes it will be sufficient to assume that all categories are locally small, therefore $\fC(A,B)$ is indeed a set, not a proper class.
The class of $\fC$-objects will be denoted by $\ob{\fC}$. The composition of $\fC$-arrows $\map f A B$ and $\map g B C$ will be denoted by $g \cmp f$.

By a \define{sequence}{sequence} in a category $\fC$ we mean a covariant functor from $\nat$ (treated as a linearly ordered category) into $\fC$. Sequences will be denoted by $\vec x$, $\vec a$, etc. Given a sequence $\map{\vec x}{\nat}{\fC}$, we denote $X_n = \vec x(n)$ and $x_n^m$ the arrow from $X_n$ to $X_m$ ($n \loe m$).
When $X \in\ob{\fC}$ is the colimit of $\vec x$, we write $X = \lim \vec x$ and we denote by $x_n^\infty$ the colimiting arrows from $X_n$ to $X$. In particular, $x_n^\infty = x_m^\infty \cmp x_n^m$ whenever $n \loe m$.

We say that $\fC$ has the \define{amalgamation property}{amalgamation property} if for every $\fC$-arrows $f$, $g$ with $\dom(f)=\dom(g)$ there exist $\fC$-arrows $f'$, $g'$ such that $f' \cmp f = g' \cmp g$. In Section~\ref{SecFourr} we shall consider a specialized variant of the amalgamation property involving transitions.

For undefined notions concerning category theory we refer to Mac Lane's monograph~\cite{MacLane}.

\subsection{Evolution systems}

An \define{evolution system}{evolution system} is a structure of the form $\Ee = \seq{\uv, \Tau, \iza}$, where $\uv$ is a category, $\iza$ is a fixed $\uv$-object, called the \define{origin}{origin}, and $\Tau$ is a class of $\uv$-arrows called \define{transitions}{transition}.
First, $\Tau$ contains identities of all $\uv$-objects. Then it is natural to require that transitions postcomposed with isomorphisms on the codomain side are transitions. Formally: $h \cmp t$ is a transition whenever $t\in \Tau$ and $h$ is an isomorphism in $\uv$ such that $h \cmp t$ is defined. It follows that all isomorphism of $\uv$ are naturally in $\Tau$.

We are interested in \define{evolutions}{evolution}, namely, sequences of the form
$$\iza = E_0 \to E_1 \to \cdots \to E_n \to \cdots$$
where each of the arrows above is a transition.
The category $\uv$ serves as the universe of discourse and the minimal assumption here is that every evolution has a colimit in $\uv$. If this is not the case, then one can always artificially add the colimits of all evolutions, making $\uv$ sufficiently big.
{While this property is connected to cocomplete categories (as in the work of Kelly \cite{Kel80}), the requirement of cocompleteness is too strong for our purposes. We only require the category to have colimits of evolutions. Moreover, we consider only monomorphisms as we want to preserve distinctness of objects along the sequence.

Contrary to the aforementioned paper of Kelly, we are not necessarily interested in the notion of reflectivity (or of a factorization system\footnote{The work \cite{Gar09} of Garner is somewhat similar in spirit. It is worth noticing that those are two different languages focused on two different notions: Garner says that "The concept of factorisation system provides us with a way of viewing a category C as a compositional product of two subcategories $L$ and $R$." Whereas we want to view our category as a family of sequences, whose colimits construct the new bigger objects.}) as our focus lies on transitions as arrows that are indecomposable, although this distinction may not be evident from the definition. It is important to note that the development of our theory and the proofs of relevant theorems do not require specifying the indecomposable or prime nature of the arrows we are interested in. However, it is worth emphasizing that the motivating examples that led us to this work are based on systems where the composition of two non-trivial transitions is no longer considered a transition. We believe that systems where the class of transitions $\Tau$ is closed under composition can be overly complicated and unpredictable to deal with,  or conversely, they may resemble categories too closely to bring anything new to the table.
}

Given a $\uv$-object $X$, we denote
$$\Tau(X) = \setof{f \in \Tau}{\dom(f) = X},$$
that is, the set of all transitions with domain $X$.
Two transitions $f, g \in \Tau(X)$ are \emph{isomorphic} if there is an isomorphism $h$ in $\uv$ such that $g = h \cmp f$.
A transition will be called \define{nontrivial}{transition!-- nontrivial} if it is not an isomorphism, and \emph{trivial} otherwise. We denote by $\Tau^+(X) = \Tau(X)\setminus \Iso(X)$ the set of all non-trivial transitions with $\dom(f) = X$.

Denote by $\cTau$ the category generated by $\Tau$, namely the smallest subcategory of $\uv$ containing $\Tau$. Note that $\cTau$ is wide, i.e., it contains all the objects of $\uv$, because all identities are transitions.
The $\cTau$-arrows will be called \define{paths}{path}. Specifically, a \emph{path} is any arrow of the form $f_0 \cmp \dots\ \cmp f_{n-1}$ where each $f_i$ is a transition.
Formally, there might be a confusion with the notion of a path, as in graph theory this should be a \emph{sequence} of arrows or transitions, while we have decided to use the name \emph{path} for a finite composition of transitions. Thus, a concrete path may have several representations as the composition of transitions (including isomorphisms)\footnote{We believe this little inaccuracy will not lead to any confusion.}. 
The \define{length}{path!-- length} of a path $f$, denoted by $\length(f)$, is the minimal number of non-trivial transitions needed to compose $f$. In particular, isomorphisms are paths of length zero and nontrivial transitions are paths of length one.
An object $X$ of $\cTau$ will be called \define{finite}{finite object} if there exist 
a path from the origin to $X$.
The minimal length of such a path will be called the \define{size}{size} of $X$, denoted by $\size(X)$.
In particular, $\size(\iza) = 0$.

Finally, we denote by $\eva{\Ee}$ the category of all finite objects with paths, namely, compositions of transitions.
Note that $\eva{\Ee}$ can also be regarded as an evolution system (a subsystem of $\Ee$), suitably restricting the class of transitions, although it would fail the minimal assumption, as typically the colimit of an evolution is not a finite object.

\pustka{Given $\Ee$ it is natural to consider two related categories. 
\begin{enumerate}
	\item The category denoted by $\eva{\Ee}$, with finite objects and paths. Note that it is a subcategory of $\Ee$ and an evolution system itself, although it would fail the minimal assumption, as typically the colimit of an evolution is not a finite object.
	\item The  category denoted by $\Eva{\Ee}$ where objects are colimits of all evolutions and as arrows take precisely the colimiting arrows, uniquely determined by taking cocones over evolutions. \textcolor{red}{?} (let us remind that every evolution is a functor $\map{F}{\nat}{\uv}$). Namely, given two evolutions $\vec x, \vec y$ such that $X_\infty = \lim \vec x$ and $Y_\infty = \lim \vec y$ we have $\Eva{\Ee}$-arrow $\map{f_\infty}{X_\infty}{Y_\infty}$. Note that every finite object is an object of $\Eva{\Ee}$, because  we can always choose identities at each step from some point on. Therefore all transitions and all paths are $\Eva{\Ee}$-arrows and we have that $\eva{\Ee} \subs \Eva{\Ee} \subs \uv$.
\end{enumerate} }

\subsection{Fundamental examples}

Let $\uv$ be a fixed category. Consider $\fK \subs \uv$ to be a wide subcategory consisting of monomorphisms. We say that a $\fK$-arrow $f$ is \define{prime}{prime} if it cannot be non-trivially decomposed, namely, whenever $f=g_1\cmp g_2$, where $g_1, g_2 \in \fK$, at least one of $g_1, g_2$ is an isomorphism.

For instance let $\uv = \pair \Qyu \leq$ with all linear mappings as arrows. As $\fK$-arrows we take all embeddings, and among them, one-point extensions would be prime arrows.  

Below we present other natural examples of evolution systems illustrating this idea.

\begin{ex}[Embeddings]\label{EXtrzijedna}
	Let $\Ef$ be a class of finite structures in a fixed first-order language consisting of relations only. It is convenient to assume $\Ef$ is closed under isomorphisms.
	Let $\ciagi{\Ef}$ denote the class of all structures of the form $\bigcup_{\ntr}X_n$, where $\ciag{X}$ is a chain in $\Ef$.
	Let $\uv$ be the category of all embeddings between structures in $\ciagi{\Ef}$. Let $\Tau$ consist of all embeddings of the form $\map f X Y$, where $Y \setminus \img fX$ is a singleton or the empty set. In other words, transitions are one-point extensions and isomorphisms.
	Finally, $\iza$ might be the empty structure.
	Clearly, $\Ee = \seq{\uv, \Tau, \iza}$ is an evolution system.
	
	Note that we can also define $\uv$ to be the category of all homomorphisms between $\ciagi{\Ef}$-objects.
	Yet another option is to consider embeddings or homomorphisms between arbitrarily large structures that can be built as unions of directed families consisting of structures from $\Ef$. One can also replace the empty structure by any (possibly large) structure, declaring it to be the origin $\iza$.
	
	Finally, one can generalize this by allowing algebraic operations in the language. Now a transition would be an embedding $\map f X Y$ such that $Y$ is generated by $\img fX \cup \sn a$ for some $a \in Y$. In this case, however, transitions may not be prime.
	A very concrete example here could be the class of all finite fields, where it is natural to define the origin $\iza$ as the $p$-element field, where $p$ is a fixed prime. By this way, the category $\eva{\Ee}$ consists of all finite fields of characteristic $p$.
\end{ex}

\begin{ex}[Quotient epimorphisms]
	Let $\Ef$ be a fixed class of finite nonempty relational structures and consider it as a category where the arrows are epimorphisms. A concrete example could be just finite sets with no extra structure.
	Define transitions to be epimorphisms $\map f X Y$ such that either $f$ is an isomorphism (a bijection) or else there is a unique $y \in Y$ with a non-trivial $f$-fiber and moreover $f^{-1}(y)$ consists of precisely two points.
	Define $\uv$ to be the opposite category, so that $f \in \uv$ is an arrow from $Y$ to $X$ if it is an epimorphism from $X$ onto $Y$.
	Then $\Ee = \seq{\uv, \Tau, \iza}$ is an evolution system, where $\iza$ is a prescribed finite structure in $\Ef$.
\end{ex}

\begin{ex}[Posets]
	Let $\poset = \pair P\loe$ be a partially ordered set with a fixed element $\perp$.
	We may assume $\perp$ is minimal and $\poset$ is well founded. In that case it is natural to say that a pair $\pair xy$ is a transition if either $x=y$ or else $x < y$ and there is no $z$ with $x < z < y$. By this way, $\poset$ becomes an evolution system with origin $\perp$. 
	Recall that every poset (in fact, every quasi-ordered set) is a category in which the arrows are pairs $\pair xy$ with $x \loe y$ and identities are pairs $\pair xx$.
	In our case, a nontrivial evolution is a sequence
	$$\perp = x_0 < x_1 < \dots < x_n < x_{n+1} < \cdots$$
	such that no $z \in P$ is strictly between two consecutive elements.
	Finite objects are those that can be reached from $\perp$ by finitely many transitions. For instance, if $\poset$ is a tree and $\perp$ is its root, then finite objects are those living on the finite levels.
\end{ex}

\begin{ex}[Directed graphs]
	Let $G$ be a directed graph with a fixed vertex $\iza$ and let $\uv$ be the free category over $G$, namely, the category of all paths in $G$. Let $\Tau$ denote the arrows of $G$, treated as paths of length one, together with all the identities of $\uv$. Then $\Ee = \seq{\uv, \Tau, \iza}$ is clearly an evolution system. The original arrows of $G$ can be reconstructed from $\Ee$ as nontrivial transitions.	
\end{ex}

\begin{ex}[Monoids]
	A monoid $\Em = \seq{M, \cdot, 1}$ is just a category with a single object $M$, whose arrows are the elements of $M$ and $\cdot$ is the composition. It can be turned into an evolution system by selecting any subset $T \subs M$ as the set of transitions.
	The only requirement is that $T$ contains all invertible elements. Evolutions may still lead to something new. A concrete example is the multiplicative monoid $\seq{\Zee \setminus\sn0, \cdot, 1}$, where perhaps the most natural choice for the transitions are all prime numbers (plus the two invertible elements $-1,1$).
	An evolution may be eventually constant one, which corresponds to a concrete natural number. Otherwise, it corresponds to a so-called \emph{super-natural} number, namely, a formal infinite product of nonnegative powers of primes
	$$\prod_{p \in \poset} p^{\al(p)},$$
	where $\poset$ denotes the set of all primes and $\al(p) \in \Nat\cup \sn \infty$. Note that $p^\infty$ means that the prime $p$ occurs infinitely many times in the evolution.
	Of course, the most complicated evolution corresponds to $\prod_{p\in \poset}p^\infty$.
	
	Non-zero integers with multiplication actually encode all embeddings of the group $\pair{\Zee}{+}$ into itself (an embedding is determined by the image of $1$). Thus, super-natural numbers correspond to sequences of self-embeddings of $\pair{\Zee}{+}$. Their colimits are torsion-free abelian groups whose all finitely generated subgroups are cyclic. The most complicated one is $\pair{\Qyu}{+}$, corresponding to $\prod_{p\in \poset}p^\infty$.
	
	{This example fits the nature of works by Hines (\cite{Hin99, Hin16}), as the main concern of both of them is only monoidal categories. We believe that our theory is much more flexible and general, as it allows constructing generic (or self-similar) objects of various nature, not necessarily of monoidal one.}
\end{ex}

\koment{\begin{ex}[Unit intervals \mrak{(draft)}]
		Let $\uv$ be a category in which the only object is the unit interval (compact space) $\unii$ and arrows are continuous surjections $\map f \unii \unii$ satisfying $f(0) = 0, f(1) = 0$. Let us recall the canonical tent map 
		$$
		t(x)= \begin{cases}
			2x,&\text{if }x\in [0, \frac{1}{2}],\\
			2-2x,&\text{if }x\in (\frac{1}{2}, 1].
		\end{cases}
		$$
		We define distorted tent maps to be non-trivial transitions, namely mappings, that are not necessarily piece-wise linear, but it holds that there exists a unique $x_0\in \unii$ for which $f(x_0) = 1$. In addition, the map is strictly increasing on $(0, x_0)$ and strictly decreasing on $(x_0, 1)$.
		
		The system is regular and determined, simply because every distorted tent map can be transformed into the canonical one using increasing (\textcolor{red}{?}) homeomorphism so every two non-trivial transitions are isomorphic. 
		
		A typical evolution is of the following form
		$$
		\begin{tikzcd}
			\theta = \unii & \unii\ar[l, "t"', twoheadrightarrow] & \unii\ar[l, "t"', twoheadrightarrow] & \ldots\ar[l, "t"', twoheadrightarrow]
		\end{tikzcd}
		$$
		The transition amalgamation property holds trivially, system is locally countable (locally finite even), so by Theorem ~\ref{THMexistencjal} (\textcolor{red}{not stated yet!}) the evolution above has the absorption property. Its colimit, which is actually an inverse limit, is Knaster's continuum, also called the bucket handle---an indecomposable compact connected Hausdorff space.
		
		\mrak{REFERENCE}
	\end{ex}
}

At this point, the Reader may have noticed that every category can be easily converted to an evolution system:

\begin{ex}[Categories]
	Fix an arbitrary category $\uv$ and fix a $\uv$-object $\iza$. Define $\Tau = \uv$.
	Then $\Ee = \seq{\uv, \Tau, \iza}$ is obviously an evolution system. It is perhaps a bit more interesting when $\iza$ is weakly initial in $\uv$ (that is, $\uv(\iza,X) \nnempty$ for every $X \in \ob{\uv}$). In any case, this shows that every category can be easily converted to an evolution system, just by fixing the origin.
\end{ex}

This example may suggest that evolution systems are so general that perhaps they give no new insight. One of our goals is to convince the readers that this is not the case.

\separator

The examples above, except the last two, perfectly fit into the language of model theory, where one-point extensions can be phrased by using quantifier-free 1-types (at least in relational languages). The two examples below show other possibilities, focusing on the idea of evolutions. The second one is fairly general, showing a natural way of building new evolution systems from old ones.

\begin{ex}
	Let $\uv$ be the category of graphs and let $\iza$ be the graph with a single vertex.
	Fix an integer $k > 0$ and declare a transition $\map t G{G'}$ to be a graph embedding such that either $G' = \img tG$ or else $G' = \img tG \cup v$, where $v$ is connected to at most $k$ vertices of $\img tG$.
	Replacing $\iza$ with a graph of at least $k$ vertices, we could change the definition of a transition, requiring the new vertex be adjacent to exactly $k$ vertices.
	In case $k = 1$, evolutions produce trees with no cycles. In the general case, evolutions produce graphs with no complete subgraph of size $> k$.	
\end{ex}

\begin{ex}\label{lewostrony}
	Fix an evolution system $\Ee = \seq{\uv, \Tau, \iza}$. Define a new system
	$$\Ee^\ddagger = \seq{\uv, \Tau^\ddagger, \iza},$$
	where $\Tau^\ddagger$ consists of all transitions $t \in \Tau$ that are left-invertible in $\uv$, that is, $g \cmp t$ is the identity for some $g \in \uv$.
	Obviously, the system $\Ee^\ddagger$ depends very much on the universe $\uv$. In particular, if $\uv$ is a category of first-order structures with embeddings, only isomorphisms are left-invertible. On the other hand, once $\uv$ consists of all homomorphisms, then left-invertibility becomes a natural property of embeddings. One can also consider a more precise variant, where instead of left-invertible arrows one considers the category of embedding-projection pairs, making the left inverses part of the structure. We refer to~\cite[Section 6]{Kub40} for details on categories of embedding-projection pairs in the context of \fra\ limits.
	
	To be more concrete, let $\uv$ be the category of graphs with all homomorphisms and let $\iza$ be the singleton graph. We assume that the graph relations are reflexive (each vertex has a loop), so that one can collapse edges. Let $\Tau$ consist, as usual, of all one-point extensions. The system $\Ee^\ddagger$ is completely different from the system of graphs $\Ee$, as evolutions produce only some graphs. For instance, cycles of length $>4$ are not finite objects in $\Ee^\ddagger$.
	This system and its relatives have actually been developed by Geschke, G{\l}\c{a}b and the first author in~\cite{GGK}.
	
	Yet another concrete case is the system $\Ee$, where $\uv$ is the category of Banach spaces (either real or complex) with non-expansive operators and $\iza$ is the trivial space.
	Evolutions in $\Ee$ lead to all separable Banach spaces, while evolutions in $\Ee^\ddagger$ lead to Banach spaces with monotone Schauder bases, see~\cite{Diestel} for details.
\end{ex}

\section{Generic evolutions}\label{SecFourr}

In this section we show that, under some natural assumptions, there exists an evolution with the absorption property and it is unique up to isomorphism of the colimits.
Next we show that such an evolution is ``the most complicated one'' and it can be characterized as a generic evolution in terms of a natural infinite game.

Most of the results are adaptations of the classical theory of universal homogeneous structures, due to \fra~\cite{Fraisse}, developed in the fifties of the last century in the context of model theory. Game-theoretic approach is due to Krawczyk and the first author~\cite{KubBM, KraKub}.
Throughout this section we assume that $\Ee = \seq{\uv, \Tau, \iza}$ is a fixed evolution system.

\subsection{Crucial properties -- amalgamation and absorption}

The first tool for proving the existence of a generic evolution is the amalgamation property, which has been considered many times in pure and applied category theory. Below we state its variant involving transitions, often more convenient to check in concrete examples.

\begin{df}[Local transition amalgamation property]
	We say that $\Ee$ has the \define{local transition amalgamation property}{local transition amalgamation property} (or lTAP) if for every two transitions $f$ and $g$ with $\dom(f)=\dom(g)=A \in \ob{\eva{\Ee}}$ there exist further transitions $f'$, $g'$ such that $f' \cmp f = g' \cmp g$, that is, the following square
	$$\begin{tikzcd}
		A \ar[r, "f"] \ar[d, "g"'] \nx B \ar[d, "f'"] \\
		C \ar[r, "g'"'] \nx D
	\end{tikzcd}$$
	is commutative.
\end{df}

It is worth noticing that the word \textit{local} in the definition above refers to the domain object $A$ being finite. It might seem more natural to define the amalgamation property for all transitions, not restricting to finite objects. The local transition amalgamation property  is relevant for the existence of a generic evolution, yet it generally does not imply the amalgamation property for arbitrary objects. While it holds true in most natural examples, this is not true in general. Nonetheless, throughout this work we will focus purely on the local transition amalgamation property. Consequently, we will no longer include the term \textit{local} and will use the acronym TAP instead of lTAP.
\pustka{cex: grupy skonczone,  amalgamacja dziala - ich granice to grupy lokalnie skonczone, amalgamacja nie dziala, see neumann and neumann}

The next simple lemma is actually quite important.

\begin{lm}\label{LMsodsod}
	TAP implies that category $\eva{\Ee}$ has the amalgamation property.
	More precisely, if $f,g$ are $\eva{\Ee}$-arrows with $\dom(f)=\dom(g)$, then there exist $\eva{\Ee}\text{-arrows}$ $f',g'$ such that $f' \cmp f = g' \cmp g$. Furthermore, $\length(f) \geq \length(g')$ and $\length(g) \geq \length(f')$.
	In particular, if $f$ is a transition then $g'$ is a transition.
\end{lm}
\begin{pf}
	Easy induction on the length of $\eva{\Ee}$-arrows. 
$$\begin{tikzpicture}
        \node (A) at (0,0) {$\bullet$};
        \node (B) at (4.48,0)  {$\bullet$};
        \node (C) at (0,-2.75)  {$\bullet$};
        \node (D) at (4.48,-2.75)  {$\bullet$};
        
        \draw[->, ppath=3] (A) -- node[above]{$f$} (B);
        \draw[->, ppath=3] (A) -- node[left]{$g$} (C);
        \draw[->, ppath=3] (B) -- node[right] {$f'$} (D);
        \draw[->, ppath=3] (C) -- node[below] {$g'$} (D); 
    \end{tikzpicture}$$
 
\end{pf}

Next key property needed for an evolution to be the most complicated one, together with this particular kind of amalgamation property, is absorption.

\begin{df}[Absorption property for transitions]
	Let $\vec u$ be an evolution. We say that $\vec u$ has the \define{absorption property for transitions}{absorption property for transitions} if for every $\ntr$, for every transition $\map t {U_n}Y$ there are $m \goe n$ and a path $\map g Y {U_m}$ such that $g \cmp t = u_n^m$.
\end{df}

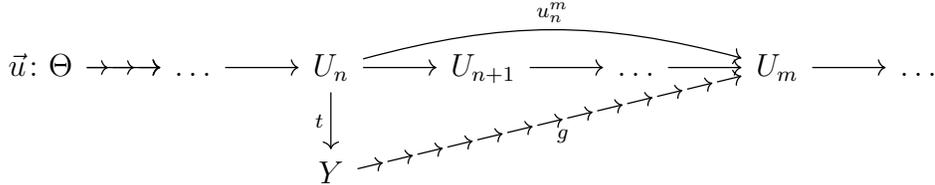
\begin{tikzcd}
	\vec u \colon \iza\ar[r, ppath=1.12] & \dots\ar[r] & {U_n}\ar[r] & {U_{n+1}} & \dots & {U_m}\ar[r] &\dots \\
	&& Y & 
	\ar[from=1-5, to=1-6]
	\ar["t"', from=1-3, to=2-3]
	\ar["{g}", swap, ppath=1.38, from=2-3, to=1-6]
	\ar[from=1-4, to=1-5]
	\ar[from=1-1, to=1-2]
	\ar["u_n^m", from=1-3, to=1-6, bend left=15]
\end{tikzcd}
\medskip

In other words, any transition going out of the evolution is ``absorbed'' at some point. Similarly, we can define the absorption property for paths. Namely consider $t$ in the previous definition to be a path instead of a transition.

\begin{df}[Path absorption property]
	An evolution $\vec u$ has the \define{path absorption property}{path absorption property} if for every $\ntr$, for every path $\map f {U_n}Y$ there exists a path $\map g Y{U_m}$ with $m \geq n$, such that $g \cmp f = u_n^m$.
\end{df}

One of the motivations of introducing the transition amalgamation property is that it gives equivalence of the absorption property for transitions and the path absorption property.

\begin{lm}\label{tratopath}
	Let $\Ee$ be an evolution system with the  transition amalgamation property and $\vec e$ be an evolution with the absorption property for transitions. Then $\vec e$ has the path absorption property.
\end{lm}	

\begin{pf}
Let the system $\Ee$ and the evolution $\vec e$ be as above and let $f_0 \in \eva{\Ee}$ be a path with $\dom(f_0)=E_{n_0}$. We want to show that $\vec e$ can absorb $f \in \Tau(E_n)$. We prove that by induction over the length of $f$. 

If $\length(f) \leq 1$, then we use the absorption property for transitions. Next, suppose that $\vec e$ can absorb paths of length $n$ and let $f \in \eva{\Ee}$ be a path with $\length(f) = n+1$ and domain $E_k$. We decompose it into $f = t \circ f_0$, where $t$ is a transition.
The following diagram visualizes the inductive step, explained below.
$$\begin{tikzpicture}
	\node (En1) at (-2,0) {$E_{k}$};
	\node (En2) at (4,0) {$E_{\ell}$};
	\node (En3) at (10,0) {$E_{m}$};
	\node (Y1) at (-0.6,-2) {$A$};
	\node (Y2) at (0.7,-4) {$B$};
	\node (X) at (5.25, -2) {$C$};
	
	\draw[->, ppath = 2.9] (En1) -- node[above] {$e_k^\ell$} (En2);
	\draw[->, ppath = 2.9] (En2) -- node[above] {$e_\ell^m$} (En3);
	
	\draw[->, ppath = 2.2] (En1) -- node[left] {$f_0$} (Y1);
	\draw[->] (En2) -- node[left] {$t'$} (X);
	\draw[->] (Y1) -- node[left] {$t$} (Y2);
	\draw[->, ppath = 3.3] (Y1) -- node[below] {$g_0$} (En2);
	\draw[->, ppath = 3.3] (Y2) -- node[below] {$g_1$} (X);
	\draw[->, ppath = 3] (X) -- node[below] {$g_2$} (En3);
	
\end{tikzpicture}
$$
The path $f_0$ is of length $n$, so by our inductive hypothesis there exist $\ell > k$ and a path $g_0$ into $E_\ell$ such that $g_0 \circ f_0 = e_k^\ell$. Now, using (the precise version of) Lemma~\ref{LMsodsod}, we obtain a path $g_1$ together with a transition $t'$ such that $g_1 \circ t = t' \circ g_0$, where $t' \in \Tau(E_k)$. Using again the absorption property for transitions, we obtain a path $g_2$ and an object $E_m$ such that $g_2 \circ t' = e_\ell^m$. Finally we have that $(g_2 \circ g_1) \circ f = e_k^m$.
\end{pf}

In Section~\ref{reg} we give an example showing that TAP is indeed crucial in this lemma. 

The last ingredient needed for the existence of an evolution with the absorption property is some kind of smallness, formally defined below.

\begin{df}
	We say that an evolution system $\Ee = \seq{\uv, \Tau, \iza}$ is \define{locally countable}{evolution system!-- locally countable} if for every finite object $X$ there is a countable subset of transitions $\Ef(X) \subs \Tau(X)$ such that for every transition $f \in \Tau(X)$ there is an isomorphism $h$ such that $h \cmp f \in \Ef(X)$.
\end{df}

We are now ready to state the main ``existential'' result.

\begin{tw}\label{THMexistencjal}
	Assume $\Ee$ is a locally countable evolution system with the transition amalgamation property.
	Then there exists a unique, up to isomorphism, evolution with the absorption property.
\end{tw}

\begin{pf}
	The existence can be proved by easy induction with a suitable bookkeeping.
	Uniqueness is a standard back and forth argument. Below we provide some details.
	
	\textbf{The existence}. We use the powerful fundamental property of the infinity: The set $\nat$ of non-negative integers can be decomposed into infinitely many infinite sets, say, $\nat = \bigcup_{\ntr}B_n$, where each $B_n$ is infinite and $B_i \cap B_j = \emptyset$ whenever $i \ne j$.
	The sets $B_n$ will be used for bookkeeping.
	
	Namely, we start at the origin $\iza$ by enumerating all transitions from $\iza$ (up to isomorphism) using the numbers from $B_0$. We use the first one to obtain the first step of our evolution $\map{e_0}{\iza}{E_1}$. We enumerate all the isomorphic types of transitions from $E_1$, using the numbers from $B_1$.
	
	At step $n$, we take the first transition $t$ from some $E_i$ with $i \loe n$ that was not considered yet and we define $\map{e_n}{E_n}{E_{n+1}}$ as the result of the amalgamation
	$$\begin{tikzcd}
		E_i \ar[rr, ppath=1.3] \ar[d, "t"'] \nx \nx E_n \ar[d, "e_n"] \\
		X \ar[rr, ppath=1.12] \nx \nx E_{n+1}
	\end{tikzcd}$$
	which is possible due to Lemma~\ref{LMsodsod}.
	Note that the horizontal arrows in the square above are compositions of transitions, namely, arrows of $\eva{\Ee}$.
	We enumerate all transitions from $E_{n+1}$ (up to isomorphism), using the set $B_{n+1}$.
	
	After infinitely many steps, we obtain an evolution
	$$\begin{tikzcd}
		\iza \ar[r] \nx E_1 \ar[r] \nx \cdots \ar[r] \nx E_n \ar[r] \nx \cdots
	\end{tikzcd}$$
	that has the absorption property, because each transition (up to isomorphism) has been taken into account.
	
	\textbf{Uniqueness}. Suppose $\vec u$, $\vec v$ are two evolutions with the absorption property.
	We start with the absorption of $\vec u$, obtaining a suitable $\eva{\Ee}$-arrow $\map {f_0}{V_0}{U_{k_0}}$. Then we use the absorption property of $\vec v$ to obtain a suitable arrow $\map{g_0}{U_{k_0}}{V_{\ell_1}}$. And so on. Instead of writing the technical details, we present the relevant (infinite) diagram
	$$\begin{tikzcd}
		U_0 \ar[r] \ar[d, "\id{\iza}"'] \nx U_{k_0} \ar[r] \ar[rd, "g_0"] \nx \cdots \ar[rrr] \nx \nx \nx U_{k_{n}} \ar[dr, "g_n"] \ar[r] \nx \cdots\\
		V_0 \ar[r] \ar[ru, "f_0"] \nx \cdots \ar[r] \nx V_{\ell_1} \ar[r] \nx \cdots \ar[r] \nx V_{\ell_{n}} \ar[ur, "f_{n}"] \ar[rr] \nx \nx V_{\ell_{n+1}} \ar[r] \nx \cdots
	\end{tikzcd}$$
	in which $U_0 = \iza = V_0$. Note that the sequences $\ciag f$ and $\ciag g$ converge to pairwise invertible arrows between the colimits $U_\infty = \lim \vec u$ and $V_\infty = \lim \vec v$, showing that $U_\infty$ and $V_\infty$ are isomorphic.	
\end{pf}

We shall see later (Example in section~\ref{reg}) that Theorem~\ref{THMexistencjal} is no longer valid when TAP is replaced by the amalgamation property of $\eva{\Ee}$.

{Our framework is somewhat similar to the one described in Leinster's aforementioned paper (\cite{Lei11}): construct a system, check whether this system has a universal solution and if so, find out how it looks like. Following his approach, a system is constructed using M-coalgebras for a suitable endofunctor of an equational system, then one looks for the universal family of spaces satisfying the equations. We are using transitions and their amalgamation to check whether there is an evolution with the absorption property and in the affirmative case we obtain the unique homogeneous object by taking the colimit of this particular sequence. To link those two approaches, Leinster says: ``Informally, an equational system is a system of equations in which each variable, representing a space, is equated to a colimit or gluing-together of the others. A universal solution of such a system is a solution with a particular universal property." This would correspond to the colimit of an evolution with the absorption property, and the equational system would be an evolution system.}

\subsection{Cofinality}

We now turn to universality, or rather \emph{cofinality}, as the term ``universal object'' has different meanings in category theory vs. model theory.
So, given a category $\fC$, we say that a $\fC$-object $U$ is \emph{cofinal} if for every $\fC$-object $X$ there is a $\fC$-arrow from $X$ to $U$. This becomes interesting when the $\fC$-arrows are some kinds of monics (in model theory, the most natural ones are embeddings).
Recall that we are working with a fixed evolution system $\Ee$ with the transition amalgamation property.

\begin{tw}[Cofinality]\label{THMkofinlnsct}
	Let $\vec u$ be an evolution with the absorption property.
	Given another evolution $\vec x$, there exists a $\uv$-arrow from $\lim \vec x$ to $\lim \vec u$.
\end{tw}

\begin{pf}
	Our goal is to obtain an infinite sequence of $\eva{\Ee}$-arrows $f_0, f_1, \dots$, so that the following infinite diagram
	$$\begin{tikzcd}
		\vec u: & U_0 \ar[r] & U_{k_1} \ar[r] & \cdots \ar[r] & U_{k_{n-1}} \ar[r] & U_{k_n} \ar[r] & \cdots \\
		\vec x: & X_0 \ar[u, "f_0"] \ar[r] & X_1 \ar[u, "f_1"] \ar[r] & \cdots \ar[r] & X_{n-1} \ar[u, "f_{n-1}"] \ar[r] & X_n \ar[u, "f_n"] \ar[r] & \cdots
	\end{tikzcd}$$
	is commutative. 
	Given $f_{n-1}$, in order to find $f_n$, we first amalgamate $f_{n-1}$ and $\map{x_{n-1}^{n}}{X_{n-1}}{X_n}$,
	obtaining a commutative square
	$$\begin{tikzcd}
		U_{k_{n-1}} \ar[r, "t"] & Y \\
		X_{n-1} \ar[u, "f_{n-1}"] \ar[r, "x_{n-1}^n"'] & X_n \ar[u, "g"']
	\end{tikzcd}$$
	in which $t$ is a transition. Next we use the absorption property of $\vec u$ so that $f_n$ is the composition of $g$ and a suitable $\eva{\Ee}$-arrow $h$ satisfying $h \cmp t = u_{k_{n-1}}^{k_n}$.
	Finally, the colimiting arrow $\map{f_\infty}{\lim \vec x}{\lim \vec u}$ witnesses that $\uv(\lim \vec x, \lim \vec u) \nnempty$.
\end{pf}

\subsection{Homogeneity}

We fix an evolution $\vec u$ with the absorption property. Let $U = \lim \vec u$. Recall that $u_n^\infty$ denotes the colimiting arrow from $U_n$ to $U$.
A \define{trajectory}{trajectory} is an arrow of the form $u_n^\infty \cmp f$, where $\ntr$ and $f \in \eva{\Ee}$.

\begin{tw}[Homogeneity]
	Assume $X$ is a finite object and $\map {i,j} X U$ are trajectories. Then there exists an automorphism $\map h U U$ such that $j = h \cmp i$.
\end{tw}

\begin{pf}
	Let us recall the infinite diagram from the proof of uniqueness (Theorem~\ref{THMexistencjal}):
	$$\begin{tikzcd}
		U_0 \ar[r] \ar[d, "\id{\iza}"'] \nx U_{k_0} \ar[r] \ar[rd, "g_0"] \nx \cdots \ar[rrr] \nx \nx \nx U_{k_{n}} \ar[dr, "g_n"] \ar[r] \nx \cdots\\
		V_0 \ar[r] \ar[ru, "f_0"] \nx \cdots \ar[r] \nx V_{\ell_1} \ar[r] \nx \cdots \ar[r] \nx V_{\ell_{n}} \ar[ur, "f_{n}"] \ar[rr] \nx \nx V_{\ell_{n+1}} \ar[r] \nx \cdots
	\end{tikzcd}$$
	Note that the same inductive arguments can be used when the sequence $\vec v$ is replaced by $\vec u$ and $\id{\iza}$ is replaced by any $\eva{\Ee}$-arrow.
	Next, we replace $U_0$ by $X$, knowing that $i = u_{k_0}^\infty 
	\cmp i'$ and $j = u_{\ell_0}^\infty \cmp j'$ for some $k_0, \ell_0$.
	Finally, we have the following infinite diagram
	$$\begin{tikzcd}
		X \ar[r, "i'"] \ar[d, "j'"'] \nx U_{k_0} \ar[r] \ar[rd, "g_0"] \nx \cdots \ar[rrr] \nx \nx \nx U_{k_{n}} \ar[dr, "g_n"] \ar[r] \nx \cdots\\
		U_{\ell_0} \ar[r] \ar[ru, "f_0"] \nx \cdots \ar[r] \nx U_{\ell_1} \ar[r] \nx \cdots \ar[r] \nx U_{\ell_{n}} \ar[ur, "f_{n}"] \ar[rr] \nx \nx U_{\ell_{n+1}} \ar[r] \nx \cdots
	\end{tikzcd}$$
	built inductively, using the absorption property. The colimiting arrow is the required automorphism.
\end{pf}

In model theory (see Example~\ref{EXtrzijedna}), homogeneity is often phrased as follows: Every isomorphism between finitely generated substructures extends to an automorphism of the whole structure. This holds when the class of structures under consideration is \emph{hereditary}, that is, closed under finitely generated substructures and isomorphisms.
On the other hand, the construction from Example~\ref{lewostrony} applied to a system of first-order structures with embeddings gives a different kind of homogeneity: Isomorphisms between substructures that could appear in generic evolutions consisting of left-invertible embeddings extend to automorphisms. Summarizing, homogeneity may sometimes be a bit far from our intuition, nevertheless it still plays a crucial role when studying the automorphism group of the colimit of a generic evolution.

\subsection{The abstract Banach-Mazur game}\label{SUBSngojweoig}

Fix an evolution system $\Ee$, for the moment with no extra properties.
We define the following game $\BMG{\Ee}{U}$ for two players, say, \emph{Eve} and \emph{Odd}. Here, $U$ is a fixed $\uv$-object.
The rules are as follows.
Eve starts the game by choosing a transition $\map {e_0}{\iza}{A_0}$.
Odd responds with a transition $\map{e_1}{A_0}{A_1}$. Eve responds with a transition $\map{e_2}{A_1}{A_2}$; Odd responds with a transition $\map{e_3}{A_2}{A_3}$. And so on.
Note that the rules for both players are identical.
After infinitely many steps, the players build an evolution $\vec a$. We say that \emph{Odd wins} if the colimit of $\vec a$ is isomorphic to $U$. Otherwise, \emph{Eve wins}. 
Such a game is well-known in model theory, see~\cite{HodgesGames} and~\cite{Jouko}.

\begin{df}
	An object $U$ is \define{generic}{generic object} if Odd has a winning strategy in the game $\BMG{\Ee}{U}$.
\end{df}

Note that a generic object is unique up to isomorphism (as long as it exists, of course).
The reason is simple: Assuming $U$, $V$ are generic, Odd can use a winning strategy aiming at $U$ while at the same time Eve can use Odd's strategy aiming at $V$. Playing such a game, the colimit is isomorphic to both $U$ and $V$.

\begin{tw}
	Assume $\Ee$ is locally countable and has the TAP.
	The following conditions are equivalent.
	\begin{enumerate}[itemsep=0pt]
		\item[{\rm(a)}] Odd has a winning strategy in $\BMG{\Ee}{U}$.
		\item[{\rm(b)}] $U$ is the colimit of an evolution with the absorption property.
	\end{enumerate}
\end{tw}

\begin{pf}
	(b)$\implies$(a) Suppose we are given a finite step of the game
	$$\begin{tikzcd}
		\iza = A_0 \ar[r, "e_1"] \nx A_1 \ar[r, "e_2"] \nx \cdots \ar[r, "e_n"] \nx A_{n}
	\end{tikzcd}$$
	and it is Odd's turn (that is, $n$ is even). He chooses $i < n$ and a transition $\map{t}{A_i}{B}$. He responds with a transition $\map{e_{n+1}}{A_n}{A_{n+1}}$ realizing an amalgamation of $t$ and the given path from $A_i$ to $A_n$ (namely, the composition $e_n \cmp \dots \cmp e_{i+1}$).
	This strategy is winning as long as Odd makes a suitable bookkeeping, so that all transitions from $A_i$s are taken into account. This is possible, since $\Ee$ is locally countable.
	
	(a)$\implies$(b) Suppose Odd has a winning strategy in $\BMG{\Ee}{U}$. Eve can use the strategy described above, so that the resulting evolution has the absorption property. This shows that $U$ is the colimit of an evolution with the absorption property.
\end{pf}

By the result above, a \emph{generic evolution} would be the one whose colimit is a generic object in the sense described above. Thus, once our evolution system has the transition amalgamation property and is locally countable, a generic evolution is precisely the one having the absorption property for paths. Without amalgamations, a generic evolution may still exist. In pure category theory this has been treated in~\cite{Kub61}; in model theory generic structures have been characterized by Krawczyk and the first author in~\cite{KraKub}.
Nevertheless, in the framework of evolution systems the following problem is open.

\begin{problem}
	Let $\Ee$ be a locally countable evolution system and let $W$ be a countable object, that is, the colimit of an evolution. Characterize the existence of a winning strategy of the second player in $\BMG{\Ee}{W}$.
\end{problem}

Once we modify the game by allowing playing with paths instead of transitions, the game becomes purely category-theoretic and then the result of~\cite{Kub61} solves the problem above. Namely, the second player has a winning strategy if and only if $\eva{\Ee}$ has the weak amalgamation property and $W$ is the colimit of an evolution with the weak absorption property. The \emph{weak amalgamation property} means: Given a finite object $F$ there is a path $\map e F{F'}$ such that for every paths $\map f {F'}X$, $\map g {F'}Y$ there are paths $\map {f'}XZ$, $\map {g'}YZ$ satisfying $f' \cmp f \cmp e = g' \cmp g \cmp e$.

$$\begin{tikzpicture}
        \node (F) at (-1.2, 1.2) {$F$};
        \node (F') at (0,0) {$F'$};
        \node (X) at (2,0) {$X$};
        \node (Y) at (0,-2) {$Y$};
        \node (D) at (2,-2) {$Z$};

        \draw[->] (F) -- node[above]{$e$} (F');
        \draw[->] (F) edge[bend left=20] node[above]{$f \circ e$} (X);
        \draw[->] (F) edge[bend right=20] node[left]{$g \circ e$} (Y);
        \draw[->] (F') -- node[above]{$f$} (X);
        \draw[->] (F') -- node[left]{$g$} (Y);
        \draw[->] (X) -- node[right]{$f'$} (D);
        \draw[->] (Y) -- node[below]{$g'$} (D);        
    \end{tikzpicture}$$

The \emph{weak absorption} property of an evolution is defined accordingly, namely given an evolution $\vec u$ we say that it has the weak absorption property if given a finite object $U_n$ of $\vec u$ there is $m>n$ such that for every path $\map f {U_m}Y$ there exist $k \geq m$ and a path $\map g Y{U_k}$, such that $g \cmp f \cmp u_n^m = u_n^k$.

$$
\begin{tikzcd}
    \vec u \colon \iza\ar[r, ppath=1.1] & {U_n} & {U_m}\ar[r] & \dots & {U_k}\ar[r, ppath=1.12] &\dots \\
	&& Y & 
	\ar[from=1-5, to=1-6]
	\ar["f "', from=1-3, to=2-3,  ppath=1.3]
	\ar["{g}", swap, ppath=1.38, from=2-3, to=1-5]
	\ar[from=1-4, to=1-5]
	\ar[from=1-1, to=1-2]
	\ar["u_n^k", from=1-2, to=1-5, bend left=15]
       \ar["u_n^m"', from=1-2, to=1-3]
\end{tikzcd}
$$

The game described above can be seen as a special case of the determinacy game, often referred to as a Gale-Stewart game (see \cite{GS}), where the players engage with natural numbers. 
More formally, consider a subset $A$ of the Baire space consisting of $\omega$-sequences of natural numbers. Then in the game $G_A$, Eve plays a natural number $a_0$, then Odd responds $a_1$, then Eve chooses $a_2$, and so on. A transition would be precisely adding one natural number. Then Eve wins the game if and only if $\langle a_{0},a_{1},a_{2},\ldots \rangle \in A$ and otherwise Odd wins. The game $G_A$ is determined if there is a winning strategy for one of the players: Eve or Odd. According to the Borel determinacy theorem from \cite{Borel}, any Gale–Stewart game is determined as long as the set $A$ is a Borel set.


\separator

The original Banach-Mazur game was invented by Mazur in the thirties of the last century, played with open intervals of the real line. It was later generalized to arbitrary topological spaces by Choquet, therefore it is also known under the name \emph{Choquet game}. We refer to the survey article~\cite{Rastislav} for detailed information on infinite topological games and to~\cite{KubBM, KraKub} for a recent study of the model-theoretic variant of the Banach-Mazur game. We also refer to the monographs~\cite{HodgesGames, Jouko} for more general infinite games in model theory.

\section{Further properties of evolution systems}\label{SecRwrtngsTrms}

The reader may have already noticed that evolution systems resemble abstract rewriting systems (see e.g.~\cite{Huet}), namely, structures of the form $\bX = \pair{X}{\to}$, where $\to$ is a binary relation, called \emph{rewriting} or \emph{reduction}. The reflexive-transitive closure of $\to$ gives a quasi-ordering of $X$, therefore $\bX$ becomes a category. The only missing ingredient is the origin, which could be any element of $X$. 

Rewriting systems are meant to model processes like reducing certain expressions (term rewriting) or processes (graph rewriting), and so on. An important feature is \emph{determination}, namely the property of an evolution saying that at each step we have only one non-trivial transition to go with. Next important notion  would be \emph{termination} which, in our language, says that every evolution ``stops'' in the sense that, from some point on, all possible transitions are isomorphisms. Another useful notion is \emph{confluence}, which corresponds precisely to the amalgamation property of the category of finite objects with paths. 

We shall now formalize these concepts in the language of evolution systems and we prove a natural extension of the important Newman's Lemma
(also called the \emph{diamond lemma}, see~\cite{Newman} and~\cite{Huet}) saying that in a terminating system local confluence implies the global one.

\begin{df}
	An evolution system $\Ee$ is \define{confluent}{confluent system} if the category $\eva{\Ee}$ of finite objects with paths has the amalgamation property.
	
	The system $\Ee$ is called \define{locally confluent}{locally confluent system} if for every finite object $X$, for every two transitions $f,g \in \Tau(X)$ there exist paths $f',g'$ such that $f' \cmp f = g' \cmp g$.
\end{df}
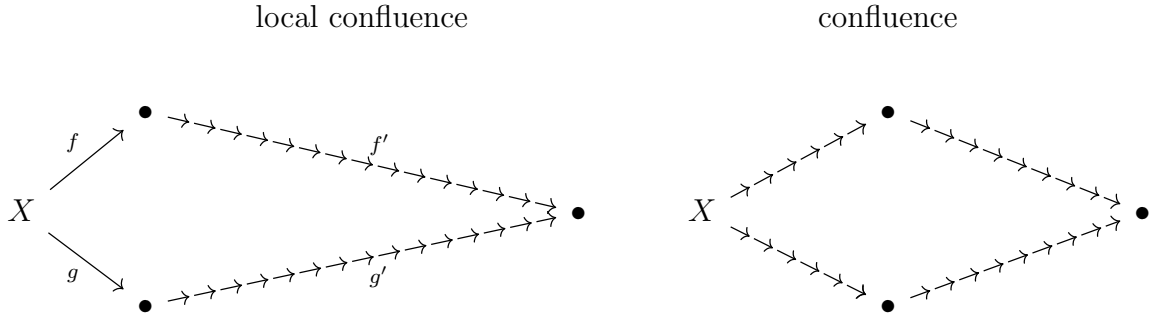
$$
\begin{tikzcd}
	&&\text{local confluence} &&& \text{confluence}\\
	& \bullet\ar["f'", drr, ppath=1.17] &&&& \bullet\ar[drr, ppath=1.12]&&\\
	X\ar["f", ur]\ar["g"', dr] &&& \bullet & X\ar[ur, ppath=1.15]\ar[dr, ppath=1.12] &&&\bullet&\\
	&\bullet\ar["g'"', urr, ppath=1.17] &&&& \bullet\ar[urr, ppath=1.12] &&
\end{tikzcd}
$$
Note that local confluence is formally weaker than the TAP, while confluence is needed for the theory of generic evolutions.
In fact, the informal introduction of this note starts with the definition of local confluence.

Before we proceed let us go back to the abstract Banach-Mazur game and point out a link with confluent evolution systems.

\begin{tw}
	Assume $\Ee$ is a confluent evolution system with an evolution $\vec u$ having the path absorption property. Its colimit $\lim \vec u$ is generic in the sense of the Banach-Mazur game. Furthermore, it is homogeneous with respect to $\eva{\Ee}$ and cofinal in the category of colimits of evolutions (cf. Theorem~\ref{THMkofinlnsct}).
\end{tw}

The proof is simply repeating the arguments from Section~\ref{SecFourr}, replacing transitions by paths.
An example in the following section shows that an locally finite confluent system with only one object may fail to have a generic evolution. A necessary and sufficient condition for the existence of an evolution with the path absorption property is a purely category-theoretic statement of being \emph{countably dominated}, see~\cite{Kub40} for details.

\subsection{Regularity}\label{reg}

Recall that a transition post-composed with an isomorphism is also a transition. In many concrete examples of evolution systems, it is natural to require that transitions composed with isomorphisms on the domain side are transitions. Formally:

\begin{df}
	An evolution system $\Ee$ is \define{regular }{evolution system!-- regular} if $t \cmp h$ is a transition whenever $t$ is a transition and $h$ is an isomorphism such that $t \cmp h$ is defined.
\end{df}

Regularity implies that isomorphisms and transitions can be amalgamated:  given an isomorphism $\map{h}{X}{\tilde X}$ and a transition $\map f XY$, there is a transition $\map {f'}{\tilde X}Y$ and an isomorphism $\map{h'}{Y}{Y}$ with $h = h' \cmp f$. Namely, $f' = f \cmp h^{-1}$ and $h' = \id Y$.

\paragraph{Evolution system that is not regular}

In the following example we present a concrete confluent evolution system with a single object and a unique nontrivial transition, which is not regular. It apparently fails to have a generic evolution. Yet it admits an evolution with the absorption property for transitions.

Let $\uv$ be the monoid of all endomorphisms of $\iza = \pair{\Qyu}{<}$. Define
$$\Tau = \aut(\Qyu,<) \cup \sn e,$$ where 
$e$ is the unique nontrivial transition, namely, a fixed embedding $\map e \Qyu \Qyu$ such that $\img e \Qyu = \Qyu \setminus\sn0$. 
Clearly, $\Ee = \seq{\uv, \Tau, \iza}$ is an evolution system.
Every evolution is an infinite sequence of self-embeddings of $\Qyu$, therefore its colimit in the category of linearly ordered sets is isomorphic to $\Qyu$. Thus evolutions are convergent in $\Ee$.

Let us look at $e$ more closely. It actually does not matter that $0$ is not in the image of $e$. What matters is the irrational number corresponding to the gap $\pair{{e^{-1}}{(-\infty,0)}}{{e^{-1}}{(0,+\infty)}}$. Let us assume that this number is $\pi$.
So $e = h \cmp i_\pi$, where $i_\pi$ is the inclusion $\Qyu \subs \Qyu \cup \sn \pi$ and $\map{h}{\Qyu \cup \sn \pi}{\Qyu}$ is a fixed isomorphism such that $h(\pi) = 0$.
Note that given any isomorphism $\map g{\Qyu\cup\sn \pi}\Qyu$, the endomorphism $f = g \cmp i_\pi$ is isomorphic to $e$, because $e = h \cmp g^{-1} \cmp f$. This indeed shows that the image of $e$ is irrelevant, while the gap that gets filled plays a significant role. Indeed, if $x$ is an irrational different from $\pi$ then $h\cmp i_x$ is not isomorphic to $e$.

\begin{prop}
	The evolution
	$$\begin{tikzcd}
		\iza \ar[r, "e"] \nx \iza \ar[r, "e"] \nx \iza \ar[r, "e"] \nx \cdots
	\end{tikzcd}$$
	has the absorption property for transitions.
\end{prop}

\begin{pf}
	Let $\map f \iza \iza$ be a transition. If $f$ is an automorphism then it is absorbed by its inverse. Otherwise, $f = e$ and then it is absorbed by the identity.
\end{pf}

Notice that paths of length $2$ are formally the form $e \cmp g$, where $g \in \aut \iza$. The automorphism $g$ can move the gap defining $\pi$ to an arbitrary gap, therefore up to isomorphism paths of length $2$ are of the form $h \cmp i_x$, where $x$ is an irrational and $h \in \aut \iza$.
Arbitrary paths are of the form $h \cmp i_S$, where $S \subs \Err \setminus \Qyu$ is finite, $\map{i_S}{\Qyu}{\Qyu \cup S}$ is the inclusion and $\map{h}{\Qyu \cup S}{\Qyu}$ is an isomorphism.

\begin{prop}
	There is no evolution with the path absorption property.
\end{prop}

\begin{pf}
	Fix an evolution
	$$\begin{tikzcd}
		\iza \ar[r, "t_0"] \nx \iza \ar[r, "t_1"] \nx \iza \ar[r, "t_2"] \nx \cdots
	\end{tikzcd}$$
	and let $f_n = t_{n-1} \cmp \dots \cmp t_0$.
	Then $\map {f_n} \iza \iza$ fills finitely many gaps in $\Qyu$ therefore the colimiting embedding $\map{f_\infty}{\iza}{\iza}$ fills countably many gaps.
	Fix an irrational $y$ that is not filled by $f_\infty$. Let $h \in \aut \iza$ be such that the inverse image of the gap defining $\pi$ is the gap defining $y$. The path $e \cmp h$ cannot be absorbed as this would lead to filling the gap defining $y$.
\end{pf}

The same argument shows that there is no generic evolution in $\Ee$. 
System $\Ee$ is not regular, because precomposing with isomorphism would formally lead to filling a different gap. For this reason, $\Ee$ does not have TAP, so the lemma~\ref{tratopath} does not hold.

Regularity will also play a significant role in the section on terminating systems. For the moment, we use it to make a short study of determined evolution systems.

\subsection{Determination}

One can have impression that evolution systems with at most one nontrivial transition at each object are rather mediocre and not interesting.
A good example here is the category of sets with transitions being one-to-one mappings ``adding'' at most one element. On the other hand, homogeneity of the limit of the generic evolution is far from being trivial, the group is $S_\infty$ --- the infinite countable permutation group.

Another example is the separable Hilbert space. It arises as the limit of an evolution on finite-dimensional Hilbert (euclidean) spaces, where the transitions are ``adding one more dimension''. Obviously, there is only one nontrivial transition from a given finite-dimensional Hilbert space, of course, up to a linear isometry.

The examples above share a common feature, formally described in the following definition.

\begin{df}
	A finite object of $\Ee$ is \emph{determined} if it admits at most one nontrivial transition, up to an isomorphism.
	An evolution system $\Ee$ is \define{determined}{evolution system!-- determined} if every finite object of $\Ee$ is determined.
	An evolution system $\Ee$ is \emph{eventually determined} if for every evolution $\vec e$ there is $n \in \nat$ such that $E_m$ is determined for every $m \geq n$.
\end{df}

\begin{claim}
	If $\Ee$ is regular, every finite object isomorphic to a determined object is also determined.
\end{claim}
\begin{proof}
	Let $X\in \ob{\Eva{\Ee}}$ be determined, $\map h X \tilde{X}$ be an isomorphism and transitions $t_0, t_1 \in \Tau^+(\tilde{X})$. By regularity there exist transitions $t_0\cmp h$ and $t_1 \cmp h$. But $X$ is determined, which means that there exist an isomorphism $\map g A B$ such that $t_1 \cmp h = g \cmp t_0 \cmp h$, but $h\in \Iso(X)$, so $t_1 = g\cmp t_0$. Therefore transitions $t_0, t_1$ are isomorphic, so $\tilde{X}$ is determined.
	$$
	
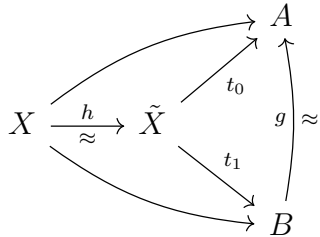
\begin{tikzcd}
		&{}&A\\
		X\ar[r,"h","\iso"']\ar[drr, bend right=15]\ar[urr, bend left=15]&\tilde{X}\ar[ur, "t_0"']\ar[dr, "t_1"]&\\
		&{}&B\ar[uu, "g", "\iso"', bend right=10]
	\end{tikzcd}
	$$
\end{proof}
\begin{tw}
	A regular determined evolution system has the TAP and admits an evolution with the absorption property.
\end{tw}

\begin{pf}
	Let us first show the TAP: Fix transitions $\map f Z X$, $\map g Z Y$, where $Z$ is a finite object. If one of $f$, $g$ is an isomorphism, we obtain an amalgamation by using its inverse and the other transition.
	If both $f$, $g$ are nontrivial, by determination there is an isomorphism $h$ with $g = h \cmp f$ and hence $h$ and the identity of the codomain of $g$ provide an amalgamation.
	
	Clearly, a determined evolution system is locally countable, therefore by Theorem~\ref{THMexistencjal} it has an evolution with the absorption property.
\end{pf}

We are now ready to state our variant of Newman's Lemma.

\begin{tw}\label{THMwiewoefqweDET}
	A regular locally confluent eventually determined evolution system is confluent.
\end{tw}

\begin{pf}
	We follow the scheme of Huet's proof of Newman's Lemma~\cite{Huet} using induction over well-founded relation. Namely, fix an evolution system $\Ee$ as in the theorem and define a quasi ordering on $\eva{\Ee}$ as follows. Let $P = \ob{\eva{\Ee}}$ and $X, Y \in P$. Declare $X \rless Y$ if there is non-trivial transition $\map t Y X$ and there exist $\tilde{X} \ne X$ and a non-trivial transition $\map{t_0} Y \tilde{X}$ such that $t$ and $t_0$ are not isomorphic. In particular it means that there are at least two non-trivial not isomorphic transitions going out of $Y$ and the codomain of one of them is precisely $X$.
	
	Note that $\Pel = \pair P \rless$ is well founded, because our system $\Ee$ is determined. 
	Note that $\rless$-minimal elements are precisely those finite objects from which there is only one non-trivial transition and they trivially admit amalgamations in $\eva{\Ee}$. Namely, let us denote by $M = \set{X \in \Pel \colon X \; \text{is} \rless \text{-minimal}}$. Let $X\in M$. Consider all possible cases
	\begin{enumerate}
		\item If both $f, g \in \Tau(X)$ are trivial, we can take their inverses to close the diagram.
		\item If $f\in \Tau(X)$ is trivial and $t\in \Tau^+(X)$, by local confluence take $f^{-1}$ and again the transition $t$
		$$\begin{tikzcd}
			X \ar[r, "t"] \ar[d, swap,  "f"] \nx Y \ar[dr, "\id Y"] \\
			\tilde{X} \ar[r, swap, "f^{-1}"] \nx X \ar[r, swap, "t"] \nx Y
		\end{tikzcd}$$
		\item Lastly, if both $t_1, t_2 \in\Tau(X)$ are non-trivial, they have to be isomorphic because $X$ is $\rless$-minimal. In particular $\cod(t_1) = \cod(t_2) = Y$, so to close the diagram it is sufficient to take $\id{Y}$.
	\end{enumerate} 
	
	Now fix an arbitrary object $Z\in \Pel \setminus M$ and two paths $\map f Z X$, $\map g Z Y$.
	Let us assume first that $g$ is a transition.
	
	If all the transitions composing $f$ are trivial then $f$ is an isomorphism, therefore it is a transition, so we amalgamate $f$, $g$ easily.
	Otherwise, $f = \tilde{f} \cmp f_0 \cmp h$, where $h$ is an isomorphism, $f_0$ is a nontrivial transition and $\tilde{f}$ is a path (possibly an identity). 
	$$\begin{tikzcd}
		Z \ar[r, "h"] \ar[d, swap,  "g"] \nx {} \ar[r, "f_0"] \nx {} \ar[r, "\tilde{f}", ppath=1.2] \nx X\\
		Y \nx \nx
	\end{tikzcd}$$
	By regularity of $\Ee$ we obtain non-trivial transiton $\tilde{f_0} = f_0 \cmp h$. Using the same argument and the second axiom of $\Tau$ we may assume, that the path $\tilde{f}$ consists of non-trivial transitions only. Note that both  $\map{\tilde{f_0}}{Z}{\tilde{Z}}$ and $g$ are non-trivial, so $\tilde{Z} \rless Z$. 
	By local confluence, there are paths $k, \ell$ such that $k \cmp \tilde{f_0} = \ell \cmp g$. 
	By inductive hypothesis (over the well-founded ordering $\rless$), there exist paths $f', k'$ such that $f' \cmp \tilde f = k' \cmp k$.
	$$\begin{tikzcd}
		Z \ar[r, "\tilde{f_0}"] \ar[d, swap,  "g"]  \nx \tilde{Z} \ar[d, "k", ppath=1.2] \ar[r, "\tilde{f}", ppath=1.12] \nx X\ar[d, "f'", ppath=1.2]\\
		Y \ar[r, "l"', ppath=1.02]\nx {} \ar[r, "k'"', ppath=1.6]\nx {} \nx {}
	\end{tikzcd}$$
	Finally, $f'$ and $k' \cmp \ell$ provide an amalgamation of $f$ and $g$.
	
	This actually completes the proof, as the general case where $g$ is a path, is settled by easy induction on its length.
\end{pf}

\subsection{Termination}

\begin{df}
	An evolution system is \define{terminating}{evolution system!-- terminating} if every evolution is eventually trivial, namely, from some point on all the transitions are isomorphisms.
	An object $X$ is \define{normalized}{normalized object} if every path from $X$ consists of isomorphisms, that is, all transitions from $X$, as well as from any object isomorphic to $X$, are trivial.
\end{df}

In other words, an evolution system is terminating if every path starting from the origin ends at a normalized object, although  such objects can exist on their own. The term ``normalized'' is inspired by the ``normal form'' in the theory of rewriting systems.

\begin{lm}
	Let $\Ee$ be a terminating confluent evolution system. Then there exists a unique (up to isomorphism) normalized object in $\Ee$.
\end{lm}

\begin{pf}
	\textbf{Existence}. Let $\Ee$ be confluent terminating evolution system. Suppose there is no normalized object. Then for every $X \in \ob{\eva{\Ee}}$ there exists an isomorphism $\map {h_X} X \tilde{X}$ such that $\Tau^+(\tilde{X}) \ne \emptyset$. This way we obtain an evolution
	$$\begin{tikzcd}
		\iza = E_0 \ar[r, "h_{E_0}", "\iso"'] \nx \tilde{E_0} \ar[r, "t_0"] \nx E_1 \ar[r, "h_{E_1}", "\iso"'] \nx \tilde{E_1}\ar[r, "t_1"] \nx \cdot 
	\end{tikzcd}$$
	where $t_i \in \Tau^+(\tilde{E_i})$ for $i \in \nat$, contradicting the termination of $\Ee$.
	
	\textbf{Uniqueness}. Suppose $U_0, U_1$ are normalized. There exists paths $\map f \iza U_0$ and $\map g \iza U_1$. System is confluent, so there exist further paths  $\map {f'} U_0 V$ and $\map {g'} U_1 V$ such that $f'\cmp f = g' \cmp g$. Note that both $f'$ and $g'$ consist of isomorphisms only, so $(g')^{-1} \cmp f' \colon U_0 \to U_1$ is an isomorphism between $U_0$ and $U_1$.
\end{pf}	

\pustka{\begin{df}[Iso-stability]
	An evolution system $\Ee$ is \define{iso-stable}{evolution system!-- iso-stable} if for every transition $\map f XY$, for every isomorphism $\map h X{\tilde X}$ there exist a transition $\map{t}{\tilde X}{\tilde Y}$ and a path $\map{f'}{Y}{\tilde Y}$ satisfying
	$f' \cmp f = t \cmp h$.
\end{df}

This property can possibly be called ``transferring transitions'' as it really says that every transition can be ``moved'' by an arbitrary isomorphism of its domain.
Let us note that typically the path $f'$ in the definition of iso-stability is an isomorphism, however in the proof of the next result we need just a path.

\begin{prop}
	A regular evolution system is \textcolor{red}{iso-stable}.
\end{prop}

\begin{pf}
	Fix a transition $\map f XY$ and and isomorphism $\map h X{\tilde X}$. By regularity, $t := f \cmp h^{-1}$ is a transition and we have $\id Y \cmp f = t \cmp h$.
\end{pf}
}

Note that a terminating evolution system is a special case of the eventually determined one. If it is also regular, then from Theorem~\ref{THMwiewoefqweDET} we obtain the following corollary.

\begin{wn}\label{THMwiewoefqwe}
	A locally confluent regular terminating evolution system is confluent.
\end{wn}

The following example shows that the regularity is necessary in Corollary~\ref{THMwiewoefqwe}.

\begin{ex}\label{EXpietSedma}
	Let $\uv$ be the category of sets, $\iza = \sn0$ and define $\Tau$ to be the class of all bijections plus two nontrivial transitions $t,s$, where $\map{t}{\sn0}{\dn01}$, $\map{s}{\sn1}{\{0,1,2\}}$ are defined by $t(0)=0$, $s(1)=1$.
	Then $\Ee = \seq{\uv, \Tau, \iza}$ is obviously a terminating evolution system.
	It is locally confluent, because each object admits at most one nontrivial transition. On the other hand, it is not confluent. Indeed, if $\map{h}{\sn0}{\sn1}$ is the (unique) bijection, then clearly the paths $t$ and $s \cmp h$ cannot be amalgamated, because the only transitions from $\dn01$ (as well as from $\{0,1,2\}$) are bijections. Clearly, $\Ee$ is not regular.
	$$\begin{tikzcd}
		\set{0} \ar[r, "t"] \ar[d, swap, "\cong"]& \set{0, 1} \ar[d, "\ncong"] \\
		\set{1} \ar[r, "s"'] & \set{0, 1, 2}
	\end{tikzcd}$$
\end{ex}

The next result does not require regularity.

\begin{tw}
	Assume $\Ee$ is a confluent terminating evolution system and $N$ is a normalized finite object. Then
	\begin{enumerate}[itemsep=0pt]
		\item[{\rm (1)}] $N$ is homogeneous.
		\item[{\rm (2)}] $N$ is cofinal in $\eva{\Ee}$.
		\item[{\rm (3)}] Every normalized finite object is isomorphic to $N$.		
	\end{enumerate}
\end{tw}

\begin{pf}
	We start with proving (3). Fix another normalized finite object $M$ and fix paths $\map f {\iza}N$, $\map g {\iza}M$. By confluence, we get paths $f'$, $g'$ with $f' \cmp f = g' \cmp g$. But $N$, $M$ are normalized, therefore $f'$, $g'$ are isomorphisms.
	Hence $(g')^{-1} \cmp f'$ is an isomorphism between $N$ and $M$.
	\pustka{$$\begin{tikzcd}
		\iza \ar[r, "f"] \ar[d, "g"']& N \ar[d, "f'"] \\
		M \ar[r, "g'"'] & M \cong N
	\end{tikzcd}$$}
	The proof of (1) is a somewhat similar: We replace $\iza$ by a finite object $A$, so we are given two paths $f$, $g$ from $A$ to $N$. The paths $f'$ and $g'$ obtained from confluence give rise to an automorphism $h = (g')^{-1} \cmp f'$ satisfying $h \cmp f = g$.
	
	Finally, (2) is proved as follows. We fix paths $\map f \iza A$, $\map g \iza N$ and, using confluence, we obtain paths $f'$, $g'$ as above. Now $g'$ must be an isomorphism, therefore $(g')^{-1} \cmp f'$ is a path from $A$ to $N$.
	\pustka{$$\begin{tikzcd}
		\iza \ar[r, "f"] \ar[d, "g"']& A \ar[d, "f'"] \\
		N \ar[r, "g'"'] & N
	\end{tikzcd}$$}
\end{pf}

The result above shows that confluent terminating evolution systems describe mathematical structures from model theory, namely, finite homogeneous structures:

\begin{ex}
	Let $N$ be a finite homogeneous first-order structure, for simplicity, let us assume the language consists of finitely many relations. Homogeneity means every partial isomorphism extends to an automorphism.
	Let $\uv$ be the category whose objects are all substructures of $N$ and arrows are homomorphisms (or just embeddings).
	Let the origin $\iza$ be the empty structure, which is an initial object of $\uv$.
	Let $\Tau$ consist of all isomorphisms and one-point extensions, as in Example~\ref{EXtrzijedna}.
	Since $N$ is homogeneous, the system $\Ee = \seq{\uv, \Tau, \iza}$ is confluent (it even has the TAP). It is terminating, as the cardinality of $N$ blocks ``long'' nontrivial paths.
	Finally, $N$ is the unique normalized object in $\Ee$.
\end{ex}

The example above looks quite trivial, however it may actually give an idea of proving the homogeneity of a given finite structure, by looking at the class of all its substructures and proving confluence.

\subsection{Termination vs. directedness}

When it comes to rewriting systems, often it is important to reach a unique normal form, which we have called a \emph{normalized} object. 

\begin{prop}
	Assume $\Ee$ is a terminating evolution system such that $\eva{\Ee}$ is directed. Then there exists a unique, up to isomorphism, normalized finite object.
\end{prop}

\begin{pf}
	Assume $N_0$, $N_1$ are normalized finite objects. Using directedness, there are paths $\map {f_0}{N_0}M$, $\map {f_1}{N_1}M$ into some finite object $M$. Since $N_0$, $N_1$ are normalized, $f_0$, $f_1$ are isomorphisms, therefore $f_1^{-1} \cmp f_0$ witness that $N_0 \iso N_1$.
\end{pf}

It turns out that another variant of Newman's Lemma is true, where local confluence is replaced by local directedness.

\begin{df}
	Fix an evolution system $\Ee$. Two finite objects $A$, $B$ will be called \define{siblings}{siblings} if there are a finite object $X$ and transitions $\map f X A$, $\map g X B$.
	
	The system $\Ee$ is \define{locally directed}{evolution system!-- locally directed} if every two siblings $A$, $B$ admit paths to a common object.
	Finally, we say that $\Ee$ is \define{directed}{evolution system!-- directed} if $\eva{\Ee}$ is directed, namely, every two finite objects admit paths to a common object.
\end{df}

Below is the announced variant of Newman's Lemma that is perhaps closer to the classical formulation.

\begin{tw}
	A regular terminating locally directed evolution system is directed.
\end{tw}

\begin{pf}
	Repeat the proof of Theorem~\ref{THMwiewoefqwe} without taking care of commutativity of the diagrams.
\end{pf}

Again, Example~\ref{EXpietSedma} above shows that regularity is necessary. Actually, a weaker variant would be sufficient, namely, the following commutativity of transitions and isomorphisms: Given a transition $f$ and an isomorphism $h$ with $\dom(f)=\dom(h)$, there should exist a transition $t$ and a path $p$ such that $\cod(t)=\cod(p)$, $\dom(t) = \cod(h)$ and $\dom(p)=\cod(f)$.

\pustka{\textcolor{red}{maybe this should be in the graph rewriting section:} Directed terminating evolution systems may be used to model the complexity of rewriting, adding ``costs'' to the transitions, assuming that isomorphisms have zero cost. A concrete example is the classical problem of minimizing the cost of computing a given product of rectangular matrices, where a transition is multiplying two adjacent matrices.
After all, every abstract or concrete rewriting system $(X,\to)$ can be transferred to an evolution system by taking the category of all paths (the free category over the simple directed graph defined by $\to$).}

\section{Graph rewriting}\label{graphrewriting}

If a state of computation can be represented as a graph, further steps can be depicted as transformation rules applied to that graph. In general graph rewriting is a widely recognised tool used to algorithmically reduce such graphs, in order to obtain as simple result as possible.

We concentrate on the historically first introduced in~\cite{first} and well-studied in literature (see for instance~\cite{BH07, BCM+99, LS05}) double-pushout (DPO) approach to graph transformation to model our processes. It is called 'DPO' because the application of a reduction to a given graph is defined via two pushout diagrams within the category of graphs and total graph morphisms: first pushout is intended to model the removal of the left-hand side of the reduction from the graph to be rewritten, and the second one the addition of the right-hand side.

More specifically, let $L, K, R$ be finite graphs such that $K \hookrightarrow L$ and $K \hookrightarrow R$. In a rule $r=\langle L, K, R\rangle $ the graph $L$ describes preconditions of a rule (pattern graph), $R$ describes postconditions (replacement graph) and $K$ contains all vertices and edges which would remain the same. To apply the rule $r$ on given graph $G$ means to find a match $L$ in $G$ and replace this part by graph $R$ obtaining a graph $H$.

The first step is to create a context graph $D=(G\setminus m(L))\cup m(K)$ using (usually\footnote{Following~\cite{HMP00} injective matching makes the approach more expressive, although some authors~\cite{Corr97} use arbitrary, possibly non-injective matching morphisms}) an injective morphism $m\colon L\to G$, which has to satisfy the \emph{dangling condition}: no edge in $G\setminus m(L)$ is incident to a vertex in $m(L\setminus K)$. 

The diagram below illustrates a double pushout (DPO)
\[\begin{tikzcd}
	L & K & R \\
	G & D & H
	\ar[hook', from=1-2, to=1-1]
	\ar[from=1-2, to=2-2]
	\ar["m"', from=1-1, to=2-1]
	\ar[hook', from=2-2, to=2-1]
	\ar[hook, from=2-2, to=2-3]
	\ar[hook, from=1-2, to=1-3]
	\ar[from=1-3, to=2-3]
\end{tikzcd}\]

In case of evolution systems, the double pushout approach can be used not only to reduce a graph but also to expand it.  Rules used to extend such a graph can vary. For instance we can start from one vertex (or even from an empty graph) and at each step add a vertex and connect it to some of already existing vertices, or even to all of them obtaining a complete graph at each step. Additionally, the rules themselves can be considerably more intricate. Let us see an example.

Let $A$ be a finite simple directed graph of the following form
\[\begin{tikzcd}
	1 & 2 & 3 \\
	& 4
	\ar[from=1-1, to=1-2]
	\ar[from=1-2, to=1-3]
	\ar[from=1-2, to=2-2]
	\ar[from=1-3, to=2-2]
\end{tikzcd}\]
Such a graph can be extended by applying the rule $r$
\begin{equation}\label{e}
	\big\{ (x_1, x_2), (x_1, x_3) \big\} \xrightarrow[\text{}]{\text{r}} \big\{ (x_1,x_3), (x_1, a), (x_2, a), (x_3, a) \big\},\tag{$*$}
\end{equation}
to visualize using DPO notation
\[\begin{tikzcd}
	x_1 & x_2 && x_1 & x_2 && x_1 & a & x_2 \\
	x_3 & \textcolor{blue}{\textbf{R}} && x_3 & \textcolor{blue}{\textbf{K}} && x_3 & \textcolor{blue}{\textbf{L}}
	\ar[from=1-7, to=1-8]
	\ar[from=1-7, to=2-7]
	\ar[from=1-9, to=1-8]
	\ar[from=1-1, to=2-1]
	\ar[from=1-1, to=1-2]
	\ar[from=1-4, to=2-4]
	\ar[color={blue}, bend left, hook', from=2-5, to=2-2]
	\ar[color={blue}, hook, bend right, from=2-5, to=2-8]
	\ar[from=2-7, to=1-8]
\end{tikzcd}\]
In our graph $A$, the subset $R=\big\{ (2, 3), (2, 4) \big\}$ matches the precondition above, so by applying the rule we obtain a graph
\[\begin{tikzcd}
	1 & 2 & 5 & 3 \\
	& 4
	\ar[from=1-1, to=1-2]
	\ar[from=1-2, to=1-3]
	\ar[from=1-4, to=1-3]
	\ar[from=1-2, to=2-2]
	\ar[from=1-4, to=2-2]
	\ar[from=2-2, to=1-3]
\end{tikzcd}\]
After one application, the $R$ still follows the rule, but now subset $\big\{ (3, 6), (3, 4) \big\}$ follows the rule as well, so we can make another two extensions
\[\begin{tikzcd}
	&& 7 \\
	1 & 2 && 5 & 6 & 3 \\
	&& 4
	\ar[from=2-1, to=2-2]
	\ar[from=2-2, to=1-3]
	\ar[from=2-2, to=3-3]
	\ar[from=3-3, to=1-3]
	\ar[from=2-4, to=1-3]
	\ar[from=3-3, to=2-4]
	\ar[from=2-4, to=2-5]
	\ar[from=3-3, to=2-5]
	\ar[from=2-6, to=2-5]
	\ar[shift left=1, from=2-6, to=3-3]
\end{tikzcd}\]
It is evident that at each step, there is the potential to make additional transitions from vertices that have already been considered. Moreover, it can happen that there appear other vertices that can also be taken into consideration. Consequently, we can apply this rule an infinite number of times, resulting in progressively larger and more complicated, yet still finite, graphs.

The following diagram illustrates a DPO using the rule (\ref{e}) of a fixed graph $G$
\pustka
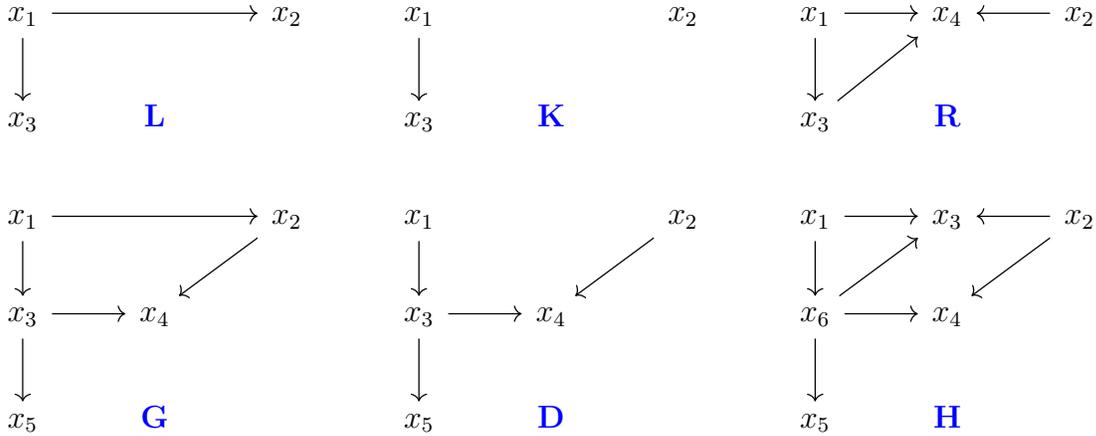
{\[\begin{tikzcd}
		{x_1} && {x_2} \\
		{x_3} & {x_4} \\
		{x_5} & \textcolor{blue}{\textbf{G}}
		\ar[from=1-1, to=1-3]
		\ar[from=1-1, to=2-1]
		\ar[from=2-1, to=2-2]
		\ar[from=1-3, to=2-2]
		\ar[from=2-1, to=3-1]
	\end{tikzcd}\]}
\[\begin{tikzcd}
	{x_1} && {x_2} & {x_1} && {x_2} & {x_1} & {x_4} & {x_2} \\
	{x_3} & \textcolor{blue}{\textbf{L}} && {x_3} & \textcolor{blue}{\textbf{K}} && {x_3} & \textcolor{blue}{\textbf{R}} \\
	{x_1} && {x_2} & {x_1} && {x_2} & {x_1} & {x_3} & {x_2} \\
	{x_3} & {x_4} && {x_3} & {x_4} && {x_6} & {x_4} \\
	{x_5} & \textcolor{blue}{\textbf{G}} && {x_5} & \textcolor{blue}{\textbf{D}} && {x_5} & \textcolor{blue}{\textbf{H}}
	\ar[from=1-1, to=1-3]
	\ar[from=1-1, to=2-1]
	\ar[from=1-4, to=2-4]
	\ar[from=1-7, to=2-7]
	\ar[from=1-7, to=1-8]
	\ar[from=1-9, to=1-8]
	\ar[from=2-7, to=1-8]
	\ar[from=3-1, to=3-3]
	\ar[from=3-3, to=4-2]
	\ar[from=4-1, to=4-2]
	\ar[from=3-1, to=4-1]
	\ar[from=4-1, to=5-1]
	\ar[from=3-4, to=4-4]
	\ar[from=4-4, to=5-4]
	\ar[from=4-4, to=4-5]
	\ar[from=3-6, to=4-5]
	\ar[from=3-7, to=3-8]
	\ar[from=3-7, to=4-7]
	\ar[from=4-7, to=4-8]
	\ar[from=4-7, to=3-8]
	\ar[from=3-9, to=3-8]
	\ar[from=3-9, to=4-8]
	\ar[from=4-7, to=5-7]
\end{tikzcd}\]

\paragraph{Translation to evolution systems}

How does it apply to evolution systems? Let $\uv$ be the category of simple directed graphs, transitions will be defined as in (\ref{e}) for one state of the graph at a time. 
\[\begin{tikzcd}
	\theta & {A_1} & {A_2} & \ldots
	\ar["{e_0}", from=1-1, to=1-2]
	\ar["{e_1}", from=1-2, to=1-3]
	\ar["{e_2}", from=1-3, to=1-4]
\end{tikzcd}\]
We shall fix the origin $\theta = A_0 = \big\{ (x_1, x_2), (x_1, x_3)\big\}$.
The first transition $e_0$ leads 
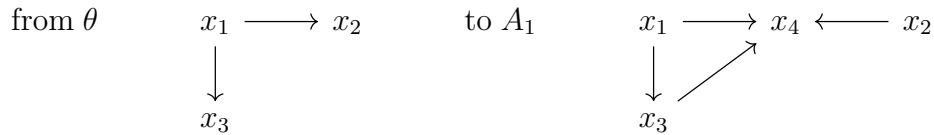
\[\begin{tikzcd}
	{\text{from} \; \theta } & {x_1} & {x_2} & {\text{to} \; A_1} & {x_1} & {x_4} & {x_2} \\
	& {x_3} &&& {x_3}
	\ar[from=1-2, to=1-3]
	\ar[from=1-2, to=2-2]
	\ar[from=1-5, to=2-5]
	\ar[from=1-5, to=1-6]
	\ar[from=1-7, to=1-6]
	\ar[from=2-5, to=1-6]
\end{tikzcd}\]
considering $x_1$, then the second transition $e_1$ leads from $A_1$ to $A_2$ again considering the same vertex 
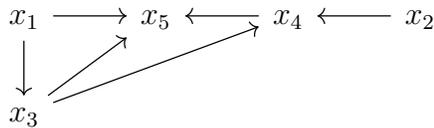
\[\begin{tikzcd}
	{x_1} & {x_5} & {x_4} & {x_2} \\
	{x_3}
	\ar[from=1-1, to=1-2]
	\ar[from=1-1, to=2-1]
	\ar[from=1-4, to=1-3]
	\ar[from=1-3, to=1-2]
	\ar[from=2-1, to=1-2]
	\ar[from=2-1, to=1-3]
\end{tikzcd}\]
Each $A_n$ is an object of $\uv$.
Now we have two possible vertices to 'fix the transition' with: $x_1$ as well as a new possibility $x_3$. As mentioned above, we can proceed infinitely many times. 

This evolution system has indeed the  amalgamation property. If two transitions $f$ and $g$ are given, both of which choose one out of $n$ possible vertices $x_n$, say $f$ expands $x_k$ and $g$ expands $x_l$, then there exist further transitions $f'$ choosing $x_l$ and $g'$ choosing $x_k$ to expand. Both compositions of transitions in fact lead to the same graph (up to an isomorphism - relabelling the vertices).  

\pustka{Consider a new notation: let $f_{x_n}$ be unique transition, which extends a vertex $x_n$ of some fixed graph $A_k$. This way we have finitely many $f_{x_i}$, and transition $\map {f_k^{k+1}} {A_k} A_{k+1}$ would accumulate all the possible $f_{x_i}$ which can be 'done' simultaneously.  This way the system is determined as there is uniquely one transition from $A_k$ to $A_{k+1}$ for every $k<\omega$, so the transition amalgamation property is trivial. Therefore it is not as interesting approach as the original one. Moreover, our goal is to define transitions to be as simple as possible, so we would rather stick to the approach of expanding the graph one vertex at the time.}

\paragraph{Other rules}
The rule mentioned above is just an example of how we can extend given graph. We can define another extension rule: 
$$
\big\{ (x_1, x_2), (x_2, x_3) \big\} \xrightarrow[\text{}]{\text{r}} \big\{ (x_1, x_2), (x_2, x_3), (x_2, x_4), (x_3, x_4) \big\}
$$
\[\begin{tikzcd}
	{x_1} & {x_2} & {x_3} && {x_1} & {x_2} & {x_3} \\
	&&&&& {x_4}
	\ar[from=1-2, to=1-3]
	\ar[from=1-1, to=1-2]
	\ar[from=1-5, to=1-6]
	\ar[from=1-6, to=1-7]
	\ar[from=1-6, to=2-6]
	\ar[from=1-7, to=2-6]
	\ar["r", shorten <=19pt, shorten >=19pt, from=1-3, to=1-5]
\end{tikzcd}\]
In this case existing edges remain the same and two more are added. Therefore at each step we have more states to consider, just as in our previous example.

In the next two examples, rewriting rules get a bit more complicated. Let us show them onlyusing diagrams:	
\pustka{$$
\big\{ (x_1, x_2), (x_1, x_3) \big\} \xrightarrow[\text{}]{\text{r}} \big\{ (x_4, x_1), (x_2, x_4), (x_2, x_3), (x_3, x_4) \big\}
$$}
\[\begin{tikzcd}
	{x_1} & {x_2} && {x_1} & {x_4} & {x_2} \\
	{x_3} &&&& {x_3}
	\ar[from=1-1, to=1-2]
	\ar[from=1-6, to=1-5]
	\ar[from=2-5, to=1-5]
	\ar[from=1-6, to=2-5]
	\ar[from=1-1, to=2-1]
	\ar[from=1-5, to=1-4]
	\ar["r_1", shorten <=19pt, shorten >=19pt, from=1-2, to=1-4]
\end{tikzcd}\]

\pustka{$$
\big\{ (x_1, x_2), (x_3, x_2) \big\} \xrightarrow[\text{}]{\text{r}} \big\{ (x_1, x_3), (x_4, x_1), (x_4, x_2), (x_4, x_3) \big\}
$$}
\[\begin{tikzcd}
	{x_1} & {x_2} && {x_1} & {x_4} & {x_2} \\
	{x_3} &&& {x_3}
	\ar[from=1-1, to=1-2]
	\ar[from=1-5, to=1-6]
	\ar[from=1-5, to=2-4]
	\ar[from=1-5, to=1-4]
	\ar["r_2", shorten <=19pt, shorten >=19pt, from=1-2, to=1-4]
	\ar[from=2-1, to=1-2]
	\ar[from=1-4, to=2-4]
\end{tikzcd}\]
Nevertheless at each step there exists pattern graph which can be extended.

It is possible to define an evolution system where transitions are constructed based on any of the extension rules mentioned earlier. The only condition is that the origin $\theta$ contains at least one appropriate subset of edges. As long as our category consists of simple directed graphs, every evolution system has the amalgamation property. 

\paragraph{Multirule approach}

It is possible to define an evolution system with amalgamation, wherein the set of transitions is defined by a finite number of rules. The amalgamation property is preserved up to relabelling the vertices as long as the rules are not contradictory. Here is an example; let $A = \big\{ (x_1, x_2), (x_2, x_3), (x_4, x_2)\big\}$ and define two distinct subsets of transitions $T_1, T_2$  
$$
f\in T_1 \iff \big\{(x_1, x_2), (x_4, x_2) \big\} \xmapsto{f} \big\{(x_1, x_4), (x_5, x_1), (x_5, x_2), (x_5, x_4) \big\}
$$
$$
g\in T_2 \iff \big\{(x_1, x_2), (x_2, x_3) \big\} \xmapsto{g} \big\{(x_1, x_2), (x_2, x_3), (x_2, x_5), (x_3, x_5) \big\}
$$
The diagram below represents a graph $A$ with transitions $f$ and $g$
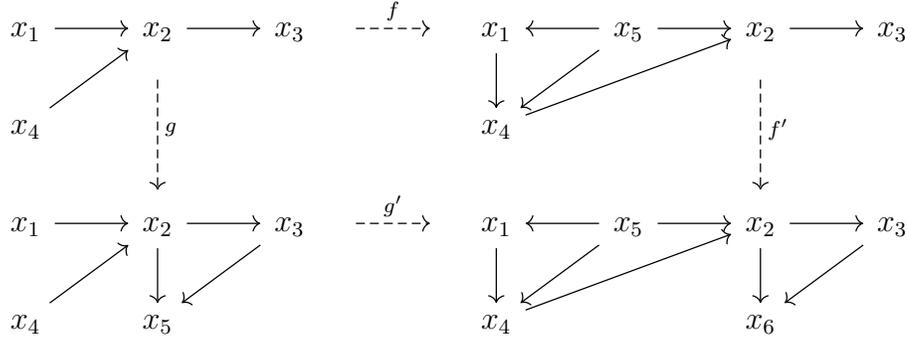
\[\begin{tikzcd}
	{x_1} & {x_2} & {x_3} && {x_1} & {x_5} & {x_2} & {x_3} \\
	{x_4} &&&& {x_4} \\
	{x_1} & {x_2} & {x_3} && {x_1} & {x_5} & {x_2} & {x_3} \\
	{x_4} & {x_5} &&& {x_4} && {x_6}
	\ar[from=2-1, to=1-2]
	\ar[from=1-1, to=1-2]
	\ar[from=1-2, to=1-3]
	\ar[from=1-5, to=2-5]
	\ar[from=1-6, to=2-5]
	\ar[from=1-6, to=1-5]
	\ar[from=1-6, to=1-7]
	\ar[from=2-5, to=1-7]
	\ar[from=3-1, to=3-2]
	\ar[from=4-1, to=3-2]
	\ar[from=3-2, to=3-3]
	\ar[from=3-2, to=4-2]
	\ar[from=3-3, to=4-2]
	\ar[from=3-6, to=4-5]
	\ar[from=3-5, to=4-5]
	\ar[from=3-6, to=3-5]
	\ar[from=3-6, to=3-7]
	\ar[from=4-5, to=3-7]
	\ar[from=1-7, to=1-8]
	\ar[from=3-7, to=4-7]
	\ar[from=3-7, to=3-8]
	\ar[from=3-8, to=4-7]
	\ar["{f'}", shorten <=10pt, shorten >=6pt, dashed, from=1-7, to=3-7]
	\ar["g", shorten <=10pt, shorten >=6pt, dashed, from=1-2, to=3-2]
	\ar["f", shorten <=14pt, shorten >=14pt, dashed, from=1-3, to=1-5]
	\ar["{g'}", shorten <=14pt, shorten >=14pt, dashed, from=3-3, to=3-5]
\end{tikzcd}\]
Note that for every $f\in T_1, g\in T_2$ it is possible to find $f' \in T_2, g'\in T_1$ such that  $f' \circ f = g' \circ g$, meaning that the composition leads to the same result.

\paragraph{Evolution systems vs. graph rewriting}

Graph rewriting primarily focuses on the simplification of graphs through the reduction of vertices and edges. The main goal is achieving termination, focusing on the final result rather than the specific way of reaching it. On the contrary, evolution systems are concerned with the expansion of structures, gradually moving towards infinity, preferably one step at a time. In this approach paths are far more relevant than the outcome. 

Terminating evolution systems may be used to model the complexity of rewriting, adding “costs” to the transitions, with isomorphisms assumed to have zero cost. A concrete example is the classical problem of minimizing the cost of
computing the product of given rectangular matrices. Besides the fact that several computational problems in linear algebra can be reduced to the computation of the product of two matrices, the complexity of matrix multiplication also arises as a bottleneck in a multitude of other computational tasks  In this case, a transition corresponds to multiplying two consecutive (or not necessarily consecutive) matrices.

Consider the multiplication $A_1\cdot A_2 \cdot \ldots \cdot A_n$ of $n$ rectangular matrices. It represents a terminating and confluent evolution system; however, the transition amalgamation property does not hold (for example given $t_k$ mutliplying $A_k\times A_{k+1}$ and $t_{k+3}$ multiplying $A_{k+3} \times A_{k+4}$, it is impossible to close the diagram with transitions). The trivial approach is to multiply fist $A_1\cdot A_2=B$ then $B\cdot A_3$ and so on. Conversely, another approach is to prioritize the multiplication of adjacent matrices of sizes $1\times n$ and $n\times 1$, resulting in a constant. This alternative strategy give us lower costs of multiplication process compared to the previous method. The standard method
for multiplying two $n \times n$ matrices uses $O(n^3)$ arithmetic operations. Additionally, a divide and conquer algorithm can be utilized to achieve an even faster process or other methods, see~\cite{matrix}. Multiple approaches can be employed to solve the problem, yielding the same outcome.

\section{Conclusions and future research}\label{SecFajw}

We believe that the concept of evolution systems will open the gate for investigating new aspects of the theory of generic objects. 
In particular, one can study the complexity of the path leading to a generic object, possibly adding some weights to the transitions.
In fact, evolution systems without any confluence properties (e.g. trees) can be of interest, too.

It should now be clear that all the examples described in Section~\ref{SecTwoo} can be formally phrased in the language of evolution systems. One of them (ribbons) actually has only the \emph{approximate} amalgamation property, namely, two transitions $f,g$ (continuous surjections of the unit interval) can be completed by transitions $f',g'$ to a square
$$\begin{tikzcd}
	\unii & \unii \ar[l, "f"'] \\
	\unii \ar[u, "g"] & \unii \ar[u, "f'"'] \ar[l, "g'"]
\end{tikzcd}$$
that commutes with a prescribed small error $\eps>0$. That is, $|f(f'(t)) - g(g'(t))| < \eps$ for every $t \in \unii$.
Formal treatment of such situations needs a category enriched over complete metric spaces (see~\cite{Kub41} for details). On the other hand, in this particular case, one can restrict to piecewise linear surjections, where the amalgamation property (known as the Mountain Climbing Theorem) holds with no errors.
Another example from Section~\ref{SecTwoo}, dealing with simplices, does not have this problem (amalgamation property just holds true), however, there are too many transitions and the absorption property holds approximately (with arbitrarily small errors). Again, in order to treat it formally, one needs to use metric-enriched categories. On the other hand, in this particular case it is possible to restrict the class of transitions so that the problem disappears. Specifically, given a simplex $\Delta_m$, a transition could be an affine surjection $\map{f}{\Delta_{m+1}}{\Delta_m}$ that is identity on $\Delta_m$ and the ``new'' vertex is mapped to a point of $\Delta_m$ that has rational barycentric coordinates.
By this way our evolution system is locally countable and the theory of Section~\ref{SecFourr} applies.

\separator

Below we indicate possible lines of further investigation, some of them inspired by recent research in the theory of generic structures.

\paragraph{Comparing evolution systems}

One of the natural lines of research is searching for tools that would allow comparing and classifying evolution systems, perhaps emphasizing on those admitting generic evolutions. For instance, given two different evolution systems, if the colimits of the respective evolutions with the absorption property are isomorphic, then perhaps being able to decide which of those systems is less/more complicated would be valuable. 

\begin{problem}
    Define a suitable notion of a rank on evolution systems that would allow to  compare the complexity of evolution systems with TAP, in particular realizing the same generic object.
\end{problem}

\paragraph{Adding probabilities} 

Another line of research is studying probabilistic approach, especially when the system is locally finite, where it is natural to impose the uniform probability. In this setting, evolution systems can model abstract stochastic processes. One can enrich an evolution system $\Ee =\triple \uv \Tau \iza$, by adding a probability measure on $\Tau(X)$ for each $X \in \ob \uv$.

\begin{problem}
    Investigate probability-enriched evolution systems. In particular, describe when the colimit of a generic evolution in this enriched setting is isomorphic to a fixed in advance object with probability one.
\end{problem}

\paragraph{Automata}

Evolution systems could possibly play some role in automata theory, where an automaton is viewed as a category (see e.g.~\cite{CoPeSt}) and transitions are the only way of moving from one state to another. One can then specify a family of objects that would be ``accepting'' the input, that is, the origin of the system, in our terminology.

\paragraph{Weak amalgamation}

Yet another possible line of research is studying the (possibly quantitative) variant of the weak amalgamation property~\cite{Kub61}, where the two transitions do not necessarily amalgamate, however, the amalgamation is possible by allowing a fixed path to take part in the amalgamation. Specifically, given a finite object $A$, there should exist a path $\map{p}{A}{A'}$ such that for every transitions $\map{f}{A'}{X}$, $\map{g}{A'}{Y}$ there are transitions $\map{f'}{X}{W}$, $\map{g'}{Y}{W}$ satisfying $f' \cmp f \cmp p = g' \cmp g \cmp p$. The length of $p$ could serve as the ``measure'' of the quality of this property.

\pustka{$$\begin{tikzpicture}[>=stealth]
        \node (F) at (-2, 2) {$A$};
        \node (F') at (0,0) {$A'$};
        \node (X) at (2,0) {$X$};
        \node (Y) at (0,-2) {$Y$};
        \node (D) at (2,-2) {$W$};

        \draw[->, ppath=2.5] (F) -- node[above]{$p$} (F');
        \draw[->] (F) edge[bend left=20] node[above]{$f \circ p$} (X);
        \draw[->] (F) edge[bend right=20] node[left]{$g \circ p$} (Y);
        \draw[->] (F') -- node[above]{$f$} (X);
        \draw[->] (F') -- node[left]{$g$} (Y);
        \draw[->] (X) -- node[right]{$f'$} (D);
        \draw[->] (Y) -- node[below]{$g'$} (D);        
    \end{tikzpicture}$$}

\begin{problem}
    Explore evolution systems, where instead of transition amaglamation property one assumes its weaker variant.
\end{problem}

\paragraph{Uncountable evolutions}

Finally, when it comes to set theory, it is natural to consider uncountable evolutions, just assuming that the universe is sequentially co-complete. In order to get generic evolutions of a given length, one needs to add an extra structure to the system, namely, a class of commutative squares of transitions, called \emph{tiles}. By this way, evolution systems generalize categories of cell complexes (see e.g.~\cite{Hovey} for details). This line of research is actually a work in progress (joint with A. Avil\'es and I. Di Liberti).
It also touches the concept of accessible categories~\cite{GabUlm} developed by Ad\'amek and Rosick\'y \cite{Ros02, Ros97, AR94} and earlier by Makkai and Par\'e~\cite{MakPar}.


\paragraph{\Kat\ functors}

The work~\cite{KubMas} of Ma\v{s}ulovi\'c and the first author explored the concept of so-called \emph{\Kat\ functors} which are actually pairs consisting of a functor $K$ and a natural transformation $\eta$ from the identity to $K$, such that additionally $K(X)$ realizes all one-point extensions of $X$. Here $X$ is a (finite or infinite) first-order structure. A \Kat\ functor can be viewed as a uniform way of constructing \fra\ limits, and one of the main applications is universality of its automorphism group.
The paper~\cite{KubMas} was written in the language of model theory, simply because pure category theory lacks the notion of one-point extension.

\begin{problem}
	Develop \Kat\ functors in the framework of evolution systems, replacing one-point extensions by transitions.
\end{problem}

Let us note that \Kat\ functors can also lead, through transfinite iterations, to uncountable structures that are homogeneous with respect to their finite substructures. We believe that in the abstract setting of evolution systems one can obtain new examples and applications, way beyond model theory.


\paragraph{Acknowledgments.}
The authors would like to thank Adam Barto\v{s}, Tristian Bice, Ivan Di Liberti, Mirna D\v{z}amonja, Paul-Andr\'e Melli\`es, Christian and Maja Pech, and Daniela Petrisan, for several valuable discussions on the topic. Special thanks are due to Christian Pech for pointing out the connections with abstract rewriting systems.


\end{document}